\documentclass[11pt,leqno,a4paper,psamsfonts]{amsart}
\usepackage[french]{babel}
\usepackage{amssymb,amsmath,amscd,enumerate,amsthm}
\usepackage[psamsfonts]{amsfonts}
\usepackage{latexsym}
\usepackage{amsbsy}
\usepackage[mathscr]{eucal}
\usepackage[latin1]{inputenc}
\usepackage{pdfsync}
\usepackage{graphicx}
\usepackage{enumerate}
\usepackage{verbatim}

\newcommand{\C}{{\mathbb{C}}}

\newcommand{\Z}{{\mathbb{Z}}}

\usepackage{color}

\newcommand{\id}{{\rm id}}
\newcommand{\comp}{{\rm Comp}}
\newcommand{\sing}{{\rm Sing}}

\newcommand{\mf}[1]{{\mathfrak{#1}}}
\newcommand{\mr}[1]{{\mathrm{#1}}}
\newcommand{\mc}[1]{{\mathcal{#1}}}
\newcommand{\mb}[1]{{\mathbb{#1}}}
\newcommand{\wt}[1]{{\widetilde{#1}}}

\newcommand{\wh}[1]{{\widehat{#1}}}

\let\ssstyle\scriptscriptstyle

\newcommand{\iso}{{\overset{\sim}{\longrightarrow}}}

\newcommand{\inte}[1]{\overset{\circ}{{#1}}}
\let\CAL=\mathcal%
\def\mathcal#1{{\CAL#1}}%

\newtheorem{teo}{Th\'{e}or\`{e}me}[section]
\newtheorem{thmm}{Th\'eor\`eme}

\newtheorem{lema}[teo]{Lemme}

\newtheorem{sublema}[teo]{Sous-Lemme}
\newtheorem{prop}[teo]{Proposition}
\newtheorem{defin}[teo]{D\'{e}finition}
\newtheorem{cor}[teo]{Corollaire}
\newtheorem{obs2}[teo]{Remarque}
\newtheorem{recap2}[teo]{R\'{e}capitulation}
\newtheorem{ex2}[teo]{Exemple}
\newenvironment{obs}{\begin{obs2}\rm}{\end{obs2}}
\newenvironment{ex}{\begin{ex2}\rm}{\end{ex2}}

\newenvironment{dem}{\begin{proof}[Preuve]}{\end{proof}}
\newenvironment{dem2}[1]{\begin{proof}[Preuve #1]}{\end{proof}}

\def\bibartp#1#2#3#4#5#6#7#8
{\bibitem{#1} {\sc #2}, {\it #3}, {#4}, {\bf t. #5}, pages { #6} \`a
{ #7}, {({#8})}}
\def\bibart#1#2#3#4#5#6
{\bibitem{#1} {\sc #2}, {\it #3}, {#4}, {\bf #5}, {({#6})}}
\def\bibliv#1#2#3#4#5
{\bibitem{#1} {\sc #2}, {\it #3}, {#4}, {({#5})}}
\def\bibaart#1#2#3#4
{\bibitem{#1} {\sc #2}, {\it #3}, {#4}}
\begin{document}
\title[Mapping Class Group d'un germe de courbe]{Mapping Class Group  d'un germe de courbe plane singuli\`ere}
\date{\today}
\author{David Mar\'{\i}n et Jean-Fran\c{c}ois Mattei}
\thanks{Le premier auteur a \'{e}t\'{e} partiellement financ\'{e} par les projets MTM2007-65122 et  MTM2008-02294 du   Ministerio de Educaci\'on y Ciencia de Espa\~{n}a / FEDER}
\address{Departament de Matem\`{a}tiques \\
Universitat Aut\`{o}noma de
Barcelona \\
E-08193 Bellaterra (Barcelona)\\
Spain} \email{davidmp@mat.uab.es}

\address{Institut de Math\'{e}matiques de Toulouse\\
Universit\'{e} Paul Sabatier\\
118, Route de Narbonne\\
F-31062 Toulouse Cedex 9, France} \email{jean-francois.mattei@math.univ-toulouse.fr}

\begin{abstract}
Nous prouvons que toute conjugaison topologique entre germes de courbes holomorphes singuli\`{e}res du plan complexe est homotope \`{a} une conjugaison qui s'\'{e}tend aux diviseurs exceptionnels des d\'{e}singu\-la\-ri\-sa\-tions minimales de ces courbes. Gr\^{a}ce \`{a} ce r\'{e}sultat, nous donnons une pr\'{e}sen\-ta\-tion explicite d'un sous-groupe d'indice fini du \emph{mapping class group} d'un  germe de courbe plane.
\end{abstract}

\maketitle

\tableofcontents

\section*{Vocabulaire}\label{chnotations}

\indent -Si $A$ est un sous-ensemble d'un espace topologique,  on d\'{e}signe par $\inte{A}$ son int\'{e}rieur, $\overline{A}$ son adh\'{e}rence et on note  $\delta A := A\setminus \inte{A}$. Si $A$ est une vari\'{e}t\'{e} \`{a} bord, son bord est not\'{e} $\partial A$.\\
\indent - Pour une courbe analytique $X$, on note
               $\mathrm{Sing}(X)$ l'ensemble de ses points singuliers et
            $\comp(X)$ la collection de ses composantes irr\'{e}duc\-tibles. Deux composantes irr\'{e}ductibles sont dites \emph{adjacentes} si elles sont distinctes et d'intersection non-vide. Le nombre $v(Y)$ de composantes de $X$ adjacentes \`{a}   $Y\in \comp(X)$   est appel\'{e} \emph{valence de $Y$}. Nous notons $\comp_k(X)$ la collection des composantes de $X$ de valence $\geq k$. 
           Une composante connexe de l'adh\'{e}rence de $X\setminus \bigcup_{Y\in \comp_3(X)}Y$ est appel\'{e}e \emph{cha\^{\i}ne (g\'{e}om\'{e}trique)} de $X$, si elle ne contient pas de composante de valence 1 ; elle est appel\'{e}e \emph{branche morte (g\'{e}om\'{e}trique) de} $X$, sinon.

\section{Introduction}\label{introduction}
Soient $(S,0)$ et $(S',0)$   deux germes de courbes holomorphes singuli\`{e}res  contenues dans $(\C^{2},0)$. Une \emph{conjugaison topologique entre $(S,0)$ et $(S',0)$}
est un germe d'hom\'{e}o\-mor\-phisme $h:(\C^{2},0)\to(\C^{2},0)$ tel que $(h(S),0)=(S',0)$. Nous dirons que la conjugaison $h$ est \emph{excellente}, si elle se rel\`{e}ve  aux r\'{e}ductions des singularit\'{e}s de $S$ et  de $S'$, d\'{e}finissant ainsi un hom\'{e}omor\-phisme entre des voisinages des diviseurs exceptionnels $\mc E_S$ et $\mc E_{S'}$ de ces r\'{e}ductions, qui :
\begin{itemize}
\item conjugue topologiquement $\mc E_S$ \`{a} $\mc E_{S'}$,
\item est compatible aux fibrations de Hopf des  composantes  de $\mc E_S$ et $\mc E_{S'}$, en dehors d'un voisinages des singularit\'{e}s du transform\'{e} total $\mc D_S$  de $S$,
\item est holomorphe au voisinage de chaque  point singulier de $\mc D_S$.
\end{itemize}
%
%
\noindent La question naturelle de  l'existence de conjugaisons excellentes est r\'{e}solue par les  r\'{e}sultats classiques de O. Zariski et M. Lejeune \cite{Zariski,Lejeune}. Les techniques de plombage
introduites par Mumford \cite{Mumford} et d\'{e}velopp\'{e}es par Neumann \cite{NeumannTAMS,EN,Neumann} permettent de pr\'{e}ciser ce probl\`{e}me et de calculer ais\'{e}ment certains invariants topologiques,  dont le \emph{groupe fondamental du compl\'{e}mentaire de $S$ : }
\begin{equation}\label{grpfondcompl}
\Gamma _S := \pi _1(\mb B_{\varepsilon }\setminus S, \cdot)\,,\quad \mb B_{\varepsilon }:=\{|z_1|^2 + |z_2|^2 \leq\varepsilon \}\,,\qquad
0<\varepsilon \ll 1\,.
\end{equation}

\medskip

L'objet de ce travail est de d\'{e}crire les ``classes d'homotopies'' des conjugaisons topologiques entre deux germes de courbes fix\'{e}s et de montrer que chaque classe contient une conjugaison excellente.

\medskip

Plus pr\'{e}cis\'{e}ment, nous disons que deux 
conjugaisons $f$, $g$ entre $(S,0)$ et $(S',0)$ sont \emph{fondamentalement \'{e}qui\-va\-lentes} et nous notons $f\asymp g$, si  
les restrictions de $f$ et $g$ \`{a} $\mathbb{B}_{\varepsilon}\setminus S$ sont homotopes en tant qu'applications \`{a} valeurs dans $\mathbb{B}_{\varepsilon'}\setminus S'$, pour $0<\varepsilon \ll \varepsilon' \ll 1$. Visiblement $\asymp$ est  une relation d'\'{e}quivalence sur l'ensemble  des conjugaisons topologiques entre $S$ et $S'$.  Notons que la   structure conique du compl\'{e}mentaire $\mathbb{B}_{\varepsilon}\setminus S$ au dessus de  $\partial \mb B_{\varepsilon}\setminus S$ et la suite suite exacte d'homotopie associ\'{e}e \`{a} sa structure de fibr\'{e} au dessus du cercle, montrent que  $\mathbb{B}_{\varepsilon}\setminus S$ est un espace d'Eilenberg-MacLane $K(\pi,1)$. La th\'{e}orie d'homotopie  classique implique alors que $f\asymp g$ si et seulement si les actions de $f$ et $g$ sur les groupes fondamentaux de $\mb B_{\varepsilon }\setminus S$ et $\mb B_{\varepsilon '}\setminus S'$ co\"{\i}ncident \`{a} automorphisme int\'{e}rieur pr\`{e}s (d\^{u} aux choix des points de base).

\medskip

Nous d\'{e}finissons un \emph{marquage de $S'$ par $S$}  comme une classe d'\'{e}quivalence d'une conjugaison entre $S$ et $S'$, pour la relation d'\'{e}quivalence fondamentale $\asymp$. Le r\'{e}sultat principal de ce travail est le suivant :

\begin{thmm}\label{thmA}
Tout marquage admet un repr\'{e}sentant excellent.
\end{thmm}

Notre construction  est enti\`{e}rement bas\'{e}e sur les r\'{e}sultats de d\'{e}composition des vari\'{e}t\'{e}s de dimension 3 de Waldhausen \cite{Waldhausen},  Jaco-Shalen \cite{JacoShalen} et Johannson \cite{Johannson}. Elle ne peut pas se d\'{e}duire des th\'{e}or\`{e}mes classiques de  Zariski-Lejeune \cite{Zariski,Lejeune}; pour s'en convaincre il suffit de consid\'{e}rer l'\'{e}ventualit\'{e}  $S=S'$, pour laquelle le r\'{e}sultat de Zariski-Lejeune est sans objet, alors que le th\'{e}or\`{e}me \ref{thmA} induit des r\'{e}sultats non-triviaux sur les automorphismes des germes de courbes.

\medskip

L'ensemble $\mc G_{S}$ de marquages d'un germe de courbe $S$ par lui-m\^{e}me est muni d'une structure de groupe par la composition. C'est l'analogue germifi\'{e} du \emph{mapping class group} des surfaces de Riemann.
La th\'{e}orie d'homotopie des espaces $K(\pi,1)$ indique que le groupe $\mc G_{S}$
se plonge dans le groupe des automorphismes ext\'{e}rieurs $\mr{Out}(\Gamma_{S})$ du groupe fondamental $\Gamma_{S}$ d\'{e}fini en (\ref{grpfondcompl}). L'image $\mr{Out}_{g}(\Gamma_{S})\subset \mathrm{Out}(\Gamma _S)$ de ce plongement  se caract\'{e}rise par la pr\'{e}servation de certaines donn\'{e}es alg\'{e}briques de $\Gamma_{S}$, qui sont de nature g\'{e}om\'{e}trique : la \emph{structure p\'{e}riph\'{e}rique munie de ses m\'{e}ridiens}, cf. la d\'{e}finition (\ref{isomgeom}), le th\'{e}or\`{e}me (\ref{pk}) et le corollaire (\ref{invgeo}).
Le sous-groupe $\mc G_{S}^{0}$ de $\mc G_{S}$ constitu\'{e} des germes d'hom\'{e}omorphismes fixant chaque composante irr\'{e}ductible de $S$, est distingu\'{e} et d'indice fini : il est \'{e}gal au  noyau du morphisme naturel de $\mc G_{S}$ dans le groupe le groupe $\mf S_{S}$ des permutations des composantes irr\'{e}ductibles de $S$. Le th\'{e}or\`{e}me pr\'{e}c\'{e}dent nous permet d'expliciter un syst\`{e}me de g\'{e}n\'{e}rateurs de $\mc G_{S}^{0}$; d\'{e}signons par $E_S : \mc B_S\to \mb C^2$ l'application de r\'{e}duction (minimale) de $S$, on a :

\begin{thmm}\label{thmB}
Il existe un \'{e}pimorphisme
$$\bigoplus\limits_{D}\mr{A}(D^{\scriptscriptstyle\bullet})\oplus\bigoplus\limits_{\mc C}\Z^{2}_{\mc C}\twoheadrightarrow \mc G_{S}^{0}\,,$$
o\`{u} : $D$ d\'{e}crit l'ensemble 
des composantes irr\'{e}ductibles de valence $v(D)\ge 3$ du diviseur total $\mc D_S:=E_S^{-1}(S)$, $\mr{A}(D^{\scriptscriptstyle\bullet})$ est le groupe d'Artin pur\footnote{i.e. le groupe de tresses pures du plan \`{a} $v(D)-1$ brins quotient\'{e} par son centre -qui est isomorphe \`{a} $\mb Z$.} de $D\cong \mb S^{2}$ point\'{e} par
$\mathrm{Sing}(\mc D_S)\cap D$,
$\mc C$ d\'{e}crit la collection $\mf C_S$ des cha\^{\i}nes de  $\mc D_S$
et $\mb Z^2_{\mc C}:=\mb Z^2$.
\end{thmm}

Remarquons que le groupe quotient $\mc G_{S}/\mc G_{S}^{0}$ correspond aux ``grandes sym\'{e}\-tries de $S$''. Notons aussi  que le graphe de la d\'{e}composition topologique de JSJ de la 3-vari\'{e}t\'{e} \`{a} bord obtenue en enlevant a la sph\`{e}re $\mb{S}^{3}_{\varepsilon}:=\partial \mb B_{\varepsilon }$ un voisinage tubulaire de l'entrelacs $S\cap\mb S_{\varepsilon }$, est constitu\'{e} de : $\comp_3(\mc D_S)$ comme ensemble de sommets  et $\mf C_S$ comme ensemble d'ar\^{e}tes. Ainsi, $\mc G_{S}^{0}$ est un groupe de graphe au sens de \cite{Serre}.
 Nous verrons sur un exemple explicite, qu'en g\'{e}n\'{e}ral l'\'{e}pimorphisme ci-dessus n'est pas un isomorphisme.\\

Le probl\`{e}me r\'{e}solu par le th\'{e}or\`{e}me A nous est apparu de fa\c{c}on naturelle dans l'\'{e}tude de la classification topologique de germes de feuilletages singuliers. Il joue un r\^{o}le cl\'{e} dans la r\'{e}solution de la conjecture de Cerveau-Sad \cite{CS,M-M,MM}.
\\

Le plan du travail est le suivant:

-Au  chapitre \ref{2} nous introduisons quelques notions sur la d\'{e}singularisation
minimale d'un germe de courbe singuli\`{e}re, ainsi que sur les tubes de Milnor (de dimension trois et quatre); elles nous permettront de pr\'{e}ciser l'\'{e}nonc\'{e} du  th\'{e}or\`{e}me principal ainsi que  la notion cl\'{e} de marquage.

-Au chapitre \ref{3} nous \'{e}tablissons les propri\'{e}t\'{e}s topologiques des tubes de Milnor qui nous seront utiles par  la suite. Ce chapitre est divis\'{e} en trois sections. Dans la premi\`{e}re nous donnons une pr\'{e}sentation du groupe fondamental du compl\'{e}mentaire d'une courbe singuli\`{e}re. Dans la deuxi\`{e}me nous pr\'{e}cisons la d\'{e}composition de Jaco-Shalen-Johannson du tube de Milnor de dimension trois, qui jouera un r\^{o}le cl\'{e} dans la preuve du th\'{e}or\`{e}me principal. Enfin, dans la troisi\`{e}me section nous \'{e}tudions les propri\'{e}t\'{e}s alg\'{e}briques de l'action d'une conjugaison topologique entre germes de courbes, sur quelques \'{e}l\'{e}ments remarquables
des groupes fondamentaux associ\'{e}s aux composantes de bord.

-Au chapitre \ref{4} nous donnons la preuve du th\'{e}or\`{e}me \ref{thmA}, structur\'{e}e en quatre sections: r\'{e}duction \`{a} la dimension trois, construction d'un hom\'{e}o\-mor\-phis\-me entre les $3$-tubes de Milnor compatible aux d\'{e}compositions de JSJ introduites dans la section \ref{3.2}, conjugaison des arbres duaux des diviseurs exceptionnels des d\'{e}singularisations, enfin extension aux $4$-tubes de Milnor.

-Finalement,  nous \'{e}tudions au chapitre \ref{5} les propri\'{e}t\'{e}s alg\'{e}briques du groupe $\mc G_{S}$ 
et nous d\'{e}montrons le th\'{e}or\`{e}me \ref{thmB}, \`{a} l'aide du th\'{e}or\`{e}me \ref{thmA} d\'{e}j\`{a} \'{e}tabli.



\section{Conjugaison de germes de courbes marqu\'{e}es}\label{2}

\subsection{D\'{e}singularisations et syst\`{e}mes locaux}

Dans tout cet article  $S$ d\'{e}signe l'intersection d'une courbe analytique de $\mb C^{2}$  avec une boule ferm\'{e} $\mb B:=\mb B_{r_0}$ de centre $0=(0,0)$ et rayon $r_{0}>0$ fix\'{e}. Nous supposons que $\mb B$ est une \emph{boule de Milnor pour $S$}, i.e. $0\in S$ et $S\setminus\{0\}$ est lisse et transverse \`{a} toutes les sph\`{e}res $\partial\mb B_{r}$, pour  $0<r\leq r_0$. Soit $E :\mc B \to \mb B$ l'application de r\'{e}duction (minimale) de $S$; nous notons $\mc D:=E^{-1}(S)$ le \emph{diviseur total}, $\mc E:=E^{-1}(0)$ le \emph{diviseur exceptionnel} et $\mc S := \overline{\mc D\setminus \mc E}$ la \emph{transform\'{e}e stricte} de $S$.
D\'{e}signons par $\comp(\mc D)$ l'ensemble de composantes irr\'{e}ductibles de $\mc D$ et par $\sing(\mc D)$ l'ensemble des points singuliers de $\mc  D$. Notons encore
$S(D):=D\cap\sing(\mc D)$.
Deux composantes $D, D'\in \comp(\mc D)$ sont dites \emph{adjacentes} si $D\neq D'$ et leur intersection $D\cap D'\neq\emptyset$; dans ce cas $D\cap D'=\{s\}\subset\sing(\mc D)$.
Nous consid\'{e}rons aussi une seconde courbe analytique $S'\ni 0$, ainsi qu'une boule de Milnor  $\mb B':=\mb B_{r'_0}$ pour $S'$; $E':\mc B'\to \mb B'$, $\mc D'$, $\mc E'$, $\mc S'$ d\'{e}signent respectivement l'application de r\'{e}duction, le diviseur total, le diviseur exceptionnel et la transform\'{e}e stricte de $S'$.
Dans tout l'article, nous adoptons les notations suivantes :
\begin{equation}\label{notationstar}
    A^{\ast} := (A\setminus S),\quad \mc A^*:=(\mc A\setminus\mc D)\,,\quad\hbox{pour}\quad A\subset \mb B\quad\hbox{et}\quad\mc A\subset\mc B\,.
\end{equation}
De m\^{e}me, si $A'\subset \mb B'$ et $\mc A'\subset \mc B'$, nous notons
$A'{}^\ast:= (A'\setminus S')$ et $\mc A'{}^\ast:= (\mc A'\setminus \mc D')$.\\

Donnons nous
pour chaque point singulier  $s\in \mr{Sing}(\mc D)$,
un syst\`{e}me de coordonn\'{e}es locales holomorphes
$(x_{s},y_{s}):\Omega_{s}\iso\mb D_{1}\times\mb D_{1}$ d\'{e}finies sur voisinage ferm\'{e} $\Omega _s$ de $s$ dans $\mc B$, \`{a} valeurs sur le polydisque ferm\'{e} $\mb D_1\times\mb D_1$, telles que $\mc D\cap \Omega_{s}=\{x_{s}y_{s}=0\}$ et $\Omega _s\cap \Omega _{s'}=\emptyset$, pour $s\neq s'$, avec $\mb D_{\varepsilon}:=\{|z|\leq\varepsilon  \}\subset \mb C$. Fixons aussi  pour chaque composante irr\'{e}ductible $D\in \comp (\mc D)$, une  fibration localement triviale en disques ferm\'{e}s donn\'{e}e par une submersion diff\'{e}rentiable $\rho_{D}:\Omega_{D}\to D$, d\'{e}finie  sur un voisinage ferm\'{e} $\Omega _D$ de $D$ dans $\mc B$. Adoptons les notations suivantes, $D$ et $D'$ d\'{e}signant encore des composantes irr\'{e}ductibles de $\mc D$ :
\begin{equation}\label{disqueslocaux}
D_s:=D\cap \Omega _s\,, \quad \hbox{pour}\quad
s\in  S(D):={\rm Sing}(\mc D)\cap D\,,
\end{equation}
et :
\begin{equation}\label{KD}
K_{D}:=\left(D\setminus\bigcup\limits_{s\in S(D)}\inte{D}_{s}\right),\quad K_{s}:=D_{s}\cup D'_{s}\,,\quad\textrm{si}\quad D\cap D'=\{s\}\,.
\end{equation}
Nous noterons aussi pour tout $X\subset\mc B$, tout  $K\subset D$ non-r\'{e}duit \`{a} point singulier  et  tout $s\in {\rm Sing}(\mc D)$,
\begin{equation}\label{XK}
X(K):=X\cap\rho_{D}^{-1}(K)\quad\hbox{et}\quad
X_{s}:=X\cap\Omega_{s}=:X(K_{s})\,.
\end{equation}
\begin{defin}\label{systlocal}
Nous dirons que la  collection $\mc L :=\left((x_s,y_s), \rho _D\right)_{s,\,D}$ est un \emph{syst\`{e}me local de $S$ sur $\mc B$}, s'il satisfait de plus les propri\'{e}t\'{e}s suivantes, pour tout $D\in \comp(\mc D)$ et  $s\in \mr{Sing}(\mc D)$ :
\begin{enumerate}[(i)]
 \item\label{tranfibMiln} la restriction de $\rho _D$ \`{a} $D$ est l'identit\'{e};
 \item si $D\subset\mc E$, alors $\rho _D$ est holomorphe sur $\rho _D^{-1}(K_D)$;
\item si $D\subset \mc S$ et $m\in D\cap \partial\mc B$, alors $\rho _D^{-1}(m)\subset \partial\mc B$.
 \item  si $z:=x_s$ ou $y_s$ d\'{e}signe la coordonn\'{e}e locale non-identiquement nulle sur $D$, alors
$z\circ\rho_{D}(m)=z(m)$, pour $|z(m)|\leq 1/2$, $m\in \Omega _D\cap\Omega _s$.
\end{enumerate}
La fibration $\rho _D$ sera appel\'{e}e ici \emph{fibration de Hopf de base \`{a} $D$}.
\end{defin}
Notons qu'alors $\rho _D$ est holomorphe au voisinage de chaque point singulier de $\mc D$ et que les composantes locales de $\mc D$ en ces points, sont des fibres de ces fibrations. Nous laissons au lecteur le soin de voir que, les cartes locales $(x_s,y_s)$ \'{e}tant donn\'{e}es, il existe des fibrations $\rho _D$ telles que $\mc L$ soit un syst\`{e}me local de $S$ sur $\mc B$.\\

\subsection{Tubes de Milnor et hom\'{e}omorphismes excellents}\label{ssectubesMilnor}

Fixons maintenant une \'{e}quation holomorphe r\'{e}duite $f$ de $S$, d\'{e}finie  sur un voisinage ouvert  de $\mb B$.
Pour tout r\'{e}el $\eta>0$, notons :
\begin{equation}\label{tubeMilnor}
T_{\eta}:=f^{-1}(\mb D_{\eta})\cap \mb B\quad \hbox{et}\quad
\mc T_{\eta}:=E^{-1}(\mc T_{\eta})\subset \mc B\,.
\end{equation}
\noindent Lorsque $\eta>0$ est assez petit, la restriction de $f$ \`{a} $T_{\eta}^*$ est une  fibration $\mc C^{\infty}$-localement triviale au dessus de $\mb D_{\eta}\setminus\{0\}$; nous dirons alors que $T_{\eta}$ et $\mc T_{\eta}$ sont des \emph{4-tubes de Milnor pour} $S$. Apr\`{e}s avoir fix\'{e}  une \'{e}quation holomorphe r\'{e}duite  de $S'$  au voisinage de $\mb B'$, nous d\'{e}finissons de la m\^{e}me mani\`{e}re les  4-tubes de Milnor pour $S'$ et les notons $T'_{\eta'}\subset \mb B'$, $\mc T'_{\eta'}\subset \mc B'$.

\begin{obs}
Si $T_{\eta}\subset \mb B$ est un 4-tube de Milnor et
$\mb B_{\varepsilon}$ est une boule ferm\'{e}e contenue dans $\inte{T}_{\eta}$, alors les inclusions $\mb B_{\varepsilon}^*\subset T_{\eta}^*\subset\mb B^*$ induisent des
isomorphismes au niveau des groupes fondamentaux.
\end{obs}

 Le syst\`{e}me local $\mc L$ \'{e}tant fix\'{e}, nous pouvons pr\'{e}ciser la topologie des 4-tubes de Milnor. Pour $\eta_0>0$ assez petit, on construit classiquement\footnote{
Par transversalit\'{e}, un  champ $\mc X_D$ poss\'{e}dant ces propri\'{e}t\'{e}s existe sur un voisinage ouvert $W_D$ de chaque $K_D$; sur  $\Omega _s$, $s\in {\rm Sing}(\mc D)$, l'existence de tels champs $\mc X_s$ r\'{e}sulte de l'homog\'{e}n\'{e}it\'{e}  de la fonction $f\circ E$. Tous ces champs se recollent en utilisant une partition de l'unit\'{e}  form\'{e}e de  fonctions   $u_D : W_D\to \mb R$ \'{e}gales \`{a} $1$ sur $\mc T_{\eta_0}(K_D)$ et $u _s : \Omega _s\to \mb R$ nulles sur $\Omega _s \cap( \cup_{D}\mc T_{\eta_0}(K_D))$, cf. \cite{Wall}.
}
un champ de vecteurs diff\'{e}ren\-tiable $\mc X$ sur  $\mc T_{\eta_0}$ s'annulant sur $\mc D$ qui, pour toute composante $D$ de $\mc D$, est tangent aux fibres de $\rho _D$ en tout point de $\mc T_{\eta_0}(K_D)$ et tel que $\mc X\cdot(f\circ E)= f\circ E$. Ce champ se projette via $E$ en un champ lipchitzien $X$ sur $T_{\eta_0}$ qui est tangent \`{a} $S$ et nul \`{a} l'origine. Les  flots de ces champs sont d\'{e}finis pour tous les temps n\'{e}gatifs. \\

 Consid\'{e}rons les 3-vari\'{e}t\'{e}s \`{a} bord suivantes, que nous appelons \emph{3-tubes de Milnor},
  \begin{equation}\label{Meta}
  M_{\eta}:=f^{-1}(\partial\mb D_{\eta})\cap\mb B\subset \partial T_{\eta}\quad \hbox{et} \quad \mc M_{\eta}:=E^{-1}(M_{\eta})\subset \partial \mc T_{\eta}\,.
\end{equation}
\`{A} l'aide des flots de $\mc X$ et de $X$, on construit facilement une r\'{e}traction par d\'{e}formation de $T_{\eta_0}^\ast$ sur $M_{\eta_0}$ -et donc aussi une r\'{e}traction par d\'{e}formation de $\mc T_{\eta_0}^\ast$ sur $\mc M_{\eta_0}$. Les propri\'{e}t\'{e}s de tangence de ces flots permettent d'\^{e}tre plus pr\'{e}cis.

\begin{prop}\label{retractT} Il existe un diff\'{e}omorphisme   $\Theta :\mc M_{\eta_0}\times ]0,\eta_0]\iso  \mc T_{\eta_0}^\ast$ tel que :
$$\Theta(\mc M_{\eta_0}\times \{\eta\}) =\mc M_{\eta}\,,\quad \Theta(\partial \mc M_{\eta_0}\times ]0,1])= \mc T_{\eta_0}'^\ast\cap \partial \mc B'\,,\quad \Theta(m,\eta_0)=m\,,$$
pour tout   $m\in \mc M_{\eta_0}$ et $0<\eta\leq\eta_0$.
De plus $D\in \comp(\mc D)$, $$\Theta(\mc M_{\eta_0}(K_D)\times ]0,\eta_0]) =\mc T_{\eta_0}^{\ast}(K_D)\,,$$
et la restriction de $\Theta$ \`{a}  $\mc M_{\eta_0}(K_D)\times ]0,\eta_0]$, se prolonge en une application diff\'{e}rentiable $\Theta_D :\mc M_{\eta_0}(K_D)\times[0,\eta_{0}] \to \mc T_{\eta_{0}}(K_D)$ v\'{e}rifiant :
$$
 \rho _D\circ \Theta_D (m,s)= \rho _{D}(m)\,,\quad \Theta_D(m, 0)=\rho _D(m)\in K_D\,.$$

\end{prop}
\noindent Ce diff\'{e}omorphisme se redescend en un diff\'{e}omorphisme
\begin{equation}\label{strproduit}
\Theta^\flat :  M_{\eta_0}\times ]0,\eta_0]\iso   T_{\eta_0}^\ast
\end{equation}
qui induit une r\'{e}traction par d\'{e}formation de $(T_{\eta_0}^\ast, T_{\eta_0}^\ast\cap\partial \mb B )$ sur $(M_{\eta_0}, \partial M_{\eta_0})$. Comme $T_{\eta_0}^\ast$ est un r\'{e}tract par d\'{e}formation de $\mb B^{\ast}$, cf.  \cite{Milnor} , $M_{\eta_0}$ est aussi un r\'{e}tract par d\'{e}formation de $\mb B^{\ast}$.
Ainsi pour $\eta$ assez petit,
les restrictions de $\rho_{D}$ \`{a} $\mc T_{\eta}(K_{D})$, sont aussi des fibrations  en disques.
\begin{obs}\label{EilenMacLane}  Les inclusions $\mb B_\varepsilon^\ast\subset T_{\eta}^\ast\subset \mb B^{\ast}$, lorsqu'elles ont lieu, induisent\footnote{Cela se voit facilement en utilisant par exemple la structure conique \cite{Milnor} du couple $(\mb B, S)$.} des isomorphismes au niveau du groupe fondamental. Comme, toujours d'apr\`{e}s \cite{Milnor}, $M_{\eta}$ fibre (par $f$) sur le cercle $\partial \mb D_\eta$, la suite exacte d'homotopie de cette fibration montre que  $M_{\eta}$ est un espace de Eilenberg-MacLane $K(\pi ,1)$.  Il en est de m\^{e}me de  $T_{\eta}^\ast$  et $\mb B^{\ast}$, qui se r\'{e}tractent sur $M_{\eta}$, ainsi que de  $\mc T_{\eta}^\ast$, $\mc B^\ast$ et  $\mc M_{\eta}$ qui leurs sont hom\'{e}omorphes.
\end{obs}
Le syst\`{e}me local $\mc L$ pour $S$ sur $\mc B$ \'{e}tant toujours fix\'{e}, consid\'{e}rons aussi un syst\`{e}me local pour  $S'$ sur $\mc B'$,
$$\mc L':= \left((x'_{s'}, y'_{s'}): \Omega '_{s'}\to \mb D_1\times\mb D_1, \; \rho '_{D'}: \Omega '_{D'}\to D'\right)_{s',\, D'}\,.$$
 Nous conservons pour $\mc L'$ les notations  (\ref{disqueslocaux}), (\ref{KD}) et (\ref{XK}) utilis\'{e}es pour $\mc L$.
\begin{defin}\label{homeoexcellents}
Un hom\'{e}omorphisme $\Phi : T_{\eta} \to T_{\eta'}'$ entre deux 4-tubes de Milnor de $S$ et de $S'$, tel que $\Phi (S)=S'$, sera dit \emph{excellent pour  $\mc L$ et $\mc L'$}, s'il se rel\`{e}ve en un hom\'{e}omorphisme $\phi  : \mc T_{\eta} \to \mc T_{\eta'}'$,  $E'\circ \phi =\Phi \circ E$, satisfaisant les propri\'{e}t\'{e}s suivantes :
\begin{enumerate}[\hphantom{hh}(a)]
\item\label{prohol} $\phi$ est holomorphe au voisinage de chaque point singulier de $\mc D$;
\item\label{proprfibre} pour toute  composante irr\'{e}ductible $D$ de $\mc D$,
    on a l'\'{e}galit\'{e} : $\phi(\mc T_{\eta}(K_{D}))=\mc T_{\eta'}'(K'_{\phi (D)})$; de plus  sur ces ensembles
 $\phi$ conjugue les  fibrations $\rho_{D}$ et $\rho'_{\phi(D)}$, i.e. $\rho '_D(\phi(m))=\phi (\rho _D(m)) $, $m\in \mc T_{\eta}(K_D)$.
\end{enumerate}
\end{defin}

\subsection{Marquages entre germes de courbes}\label{marcquage_de_courbe} Classiquement la structure conique de $\mb B^\ast$, qui sera pr\'{e}cis\'{e}e en (\ref{reductrois}), induit une r\'{e}traction par d\'{e}forma\-tion de $\mb B^\ast$ sur toute boule ferm\'{e}e plus petite $\mb B_{\varepsilon}^\ast\subset {\mb B}^\ast$. Si $T_{\eta}\subset \mb B$ est un 4-tube de Milnor contenant
$\mb B_{\varepsilon}$, les inclusions $\mb B_{\varepsilon}^*\subset T_{\eta}^*\subset\mb B^*$ induisent donc des
isomorphismes au niveau des groupes fondamentaux. Toute application continue de l'un de ces ensembles vers $\mb B'{^{\ast}}$ d\'{e}finit ainsi un morphisme du groupe fondamental de $\mb B^{\ast}$ dans celui de $\mb B'{^{\ast}}$ . Plus pr\'{e}cis\'{e}ment consid\'{e}rons l'ensemble $\mf C(\mb B^\ast, \mb B'{}^\ast)$ des applications continues $F : U \to \mb B'$, o\`{u} $ U$ est un sous-ensemble connexe par arc de $\mb B^{\ast}$, tel que l'application d'inclusion $i_U :U\hookrightarrow \mb B^{\ast}$ induit un isomorphisme $i_{U\ast} :\pi _1(U ,\,p)\iso\pi _1(\mb B^*,\, p)$. Notons alors :
$$
    \underline{F}_{\ast} := F_\ast\circ(i_U)^{-1}: \pi _1(\mb B^{\ast}, p)\to \pi _1(\mb B'{}^{\ast}, F(p))\,.
$$
\begin{defin}\label{equivfond} Nous dirons que deux \'{e}l\'{e}ments $F : U\to \mb B'{^{\ast}}$ et $G : V\to \mb B'{^{\ast}}$ de $\mf C(\mb B^\ast, \mb B'{}^\ast)$ sont \emph{fondamentalement \'{e}quivalents} et nous noterons $F\asymp G$, si pour tout chemin $\alpha$ trac\'{e} dans $\mb B^*$ d'origine un point $p$ de $U$ et d'extr\'{e}mit\'{e} un point $q$ de $V$, il existe un chemin $\alpha'$ trac\'{e} dans $\mb B'^*$ d'origine $F(p)$ et d'extr\'{e}mit\'{e} $G(q)$, tel que
\begin{equation}\label{FG}
\alpha'_{*}\circ\underline{F}_{*}=\underline{G}_{*}\circ\alpha_{*},
\end{equation}
o\`{u} $\alpha_{*}:\pi_{1}(\mb B^*,p)\to\pi_{1}(\mb B^*,q)$ et $\alpha'_{*}:\pi_{1}(\mb B'^*,F(p))\to\pi_{1}(\mb B'^*,G(q))$ sont les isomorphismes  induits par $\alpha$ et $\alpha'$.
\end{defin}
\noindent On v\'{e}rifie facilement que $\asymp$ est bien une relation d'\'{e}quivalence dans $\mf C(\mb B^\ast, \mb B'{}^\ast)$ et que $F\asymp G$ d\`{e}s qu'il existe existe un couple de chemins $(\alpha , \alpha ')$ satisfaisant l'\'{e}galit\'{e} (\ref{FG}).
\begin{defin}
Une classe d'\'{e}quivalence $\mf f$ de $\asymp$ s'appellera \emph{marquage de $S'$ par $S$}, s'il existe un voisinage ouvert $U$ de l'origine dans $\mb B$ et un repr\'{e}sentant $\breve F : U^{\ast}\to \mb B'{^{\ast}}$ de $\mf f$, qui se prolonge en un hom\'{e}omorphisme $F : U\iso F(U)\subset \mb B'$ v\'{e}rifiant $F(S\cap U)= S'\cap F(U)$ et
pr\'{e}servant les orientations\footnote{Si $S=S'$ est donn\'{e}e par une \'{e}quation r\'{e}elle alors $F(x,y)=(\bar x,\bar y)$ pr\'{e}serve l'orientation de l'espace ambiante mais renverse l'orientation de $S$.} de $U,U'$ et $S,S'$ en tant que courbes complexes.
\end{defin}

\vspace{1em}

\begin{itemize}
\item[] \textsl{Dor\'{e}navant tous les hom\'{e}omorphismes conjuguant deux germes de courbes que nous consid\'{e}rons sont suppos\'{e}s pr\'{e}server l'orientation de l'espace ambiant et celles des courbes holomorphes.}
\end{itemize}

\vspace{2em}

Visiblement deux hom\'{e}omorphismes conjuguant $S$ \`{a} $S'$ sur des voisinages de l'origine d\'{e}finissent  le m\^{e}me marquage (de $S'$ par $S$) s'ils ont m\^{e}me germe \`{a} l'origine. Nous parlerons donc de \emph{germe d'hom\'{e}o\-morphisme repr\'{e}sentant un marquage}.
Comme d'apr\`{e}s (\ref{EilenMacLane}), $\mb B_{\epsilon}^*$ est un espace $K(\pi,1)$, un th\'{e}or\`{e}me classique\footnote{Cf. par exemple \cite{Whitehead} corollaire (4.4), page 226.} de topologie alg\'{e}brique donne la caract\'{e}risation  suivante.
\begin{prop}\label{fondequivhomot}
Deux germes d'hom\'{e}o\-morphismes au voisinage de  $0\in\mb C^2$ conjuguant les germes \`{a} l'origine de $S$ et de $S'$, repr\'{e}sentent le m\^{e}me marquage  si et seulement si pour $\varepsilon>0$ assez petit, ils induisent   des applications de  $\mb B_{\varepsilon}^*$  dans $\mb B'{^{\ast}}$, qui sont homotopes.
\end{prop}
\noindent Cela nous am\`{e}ne \`{a} poser la \\

\noindent{\bf Question. } Est-ce que deux germes d'hom\'{e}omorphismes $h_0$, $h_1 : (\mb C^2,0)\to (\mb C^2,0)$ tels que $h_{i}(S,0)=(S',0)$, $i=0,1$, d\'{e}finissent le m\^{e}me marquage si et seulement s'il existe une germe d'hom\'{e}omorphisme $H : (\mb C^3, I)\to (\mb C^3, I)$ le long du compact $I:= 0\times0\times[0,1]$, qui s'\'{e}crit $H(x,y, t)=(H_t(x,y),t)$, v\'{e}rifie $H_0=h_0$, $H_1=h_1$ et tel que les germes de $H(S\times[0,1])$ et de $S'\times[0,1]$ le long de $I$, soient \'{e}gaux?\\

\noindent Le r\'{e}sultat principal de ce travail est le th\'{e}or\`{e}me suivant.
\begin{teo}\label{teomarquagexcellent}
Soient  $\mc L=((x_s,y_s),\rho_{D})_{s,\,D}$ et $\mc L'=((x'_{s'},y'_{s'}),\rho'_{D'})_{s',\,D'}$ deux syst\`{e}mes locaux
de $S$ et $S'$ sur $\mc B$ et $\mc B'$ respectivement et
 $h:\mb B_{\varepsilon}\iso h(\mb B_{\varepsilon})\subset\mb B'$ un hom\'{e}omorphisme tel que $h(S\cap \mb B_{\varepsilon})=S'\cap h(\mb B_{\varepsilon}) $.
Alors il existe un hom\'{e}omorphisme $\Phi: T_{\eta}\iso T_{\eta'}'$, $\Phi (S)=S'$, qui est  excellent pour $\mc L$ et $\mc L'$, et tel que les restrictions $h_{|\mb B_{\varepsilon}^\ast}$ et  $\Phi _{|T_{\eta}^\ast} : T_{\eta}^\ast\to T'^{\ast}_{\eta'}$ sont fondamentalement \'{e}quivalentes.
\end{teo}

\noindent En d'autres termes, on obtient le th\'{e}or\`{e}me A de l'introduction :\\

\begin{itemize}\item[] \textit{Tout marquage de $S'$ par $S$ peut \^{e}tre repr\'{e}sent\'{e} par un hom\'{e}omorphisme excellent entre deux 4-tubes de Milnor.}
\end{itemize}

\section{Topologie des tubes de Milnor}\label{3}

Avant de commencer la preuve du th\'{e}or\`{e}me  \ref{teomarquagexcellent}, nous \'{e}tablissons dans cette section les  propri\'{e}t\'{e}s topologiques des tubes de Milnor que nous utiliserons par la suite.

\subsection{Groupe fondamental et homologie} Nous allons expliciter une pr\'{e}\-sen\-ta\-tion du groupe fondamental $\Gamma$ de $T_{\eta}^*$. Pour cela, rappelons que l'arbre dual $\mb A$ de la d\'{e}singularisation de $S$ est constitu\'{e} d'un sommet pour chaque \'{e}l\'{e}ment $D\in\comp(\mc D)$ et d'une ar\^{e}te joignant les sommets correspondants \`{a} $D$ et $D'$, pour chaque singularit\'{e} $s\in\sing(\mc D)$, point d'intersection de $D$ et $D'$.\\

Fixons un syst\`{e}me local $\mc L$ pour $S$  ainsi qu'une immersion topologique $j$ d'une r\'{e}alisation g\'{e}om\'{e}trique $|\mb A|$ de $\mb A$ dans $\mc T_{\eta}^*$, telle que :
\begin{itemize}
\item pour chaque $D\in\comp(\mc D)$, $j^{-1}(\mc T_{\eta}^*(K_{D}))$ contient un seul somment $\mathbf{s}_D$ de $\mb A$, qui est celui
    associ\'{e} \`{a} $D$ et de plus $\rho _D\circ j$ est un plongement sur un voisinage de $\mathbf{s}_D$;
\item pour chaque $s\in\sing(\mc D)$, $j^{-1}(\mc T_{\eta}^*(K_{s}))$ est contenu dans une seule ar\^{e}te, qui est celle associ\'{e}e \`{a} $s$. Nous supposons aussi, ce qui est toujours possible, que le point de coordonn\'{e}es $(x_{s},y_{s})=(1,1)$ appartient \`{a} $j(\mb A)$.
\end{itemize}
Comme $j(\mb A)$ est contractile, on a un isomorphisme canonique entre les groupes $\pi_{1}(\mc T_{\eta}^*,j(\mb A))$ et $\Gamma$ qu'on nous n'expliciterons pas.

\begin{defin}
Pour chaque $D\in\comp(\mc D)$ consid\'{e}rons l'\'{e}l\'{e}ment $\mf c_{D}\in\Gamma$ repr\'{e}sent\'{e} par le lacet simple $\rho_{D}^{-1}(\rho_{D}(j(\mathbf{s}_{D})))\cap \mc M_{\eta}$, orient\'{e} comme le bord d'une courbe holomorphe de $\mc T_{\eta}^*$.
\end{defin}

\begin{obs}
Soit $s\in\sing(\mc D)$ le point d'intersection de $D$ et $D'\in\comp(\mc D)$. Supposons que l'on a : $D\cap \Omega_{s}=\{x_{s}=0\}$ et $D'\cap\Omega_{s}=\{y_{s}=0\}$. Alors $\mf c_{D}$ et $\mf c_{D'}$
sont les classes d'homotopie des lacets
$(x_{s},y_{s})=(e^{2i\pi t},1)$ et $(x_{s},y_{s})=(1,e^{2i\pi t})$ respectivement.
\end{obs}

\begin{prop}
Le groupe fondamental $\Gamma$ admet comme syst\`{e}me de g\'{e}n\'{e}\-ra\-teurs $\{\mf c_{D}\}_{D\in\comp(\mc D)}$, avec les relations donn\'{e}es par les familles\footnote{Dans le produit $\prod\limits_{D'\in\comp(\mc D)}\mf c_{D'}^{(D',E)}$, l'ordre est celui donn\'{e} par l'ordre cyclique des ar\^{e}tes de la projection $\rho_{E}\circ j(\mr{star}(\mathbf{s}_{E}))$ de l'\'{e}toile de $\mathbf{s}_E$, dans la composante $E$, qui est de dimension r\'{e}elle deux.}

\begin{equation}\label{relgamma}\prod\limits_{D'\in\comp(\mc D)}\mf c_{D'}^{(D',E)}=1,\quad [\mf c_{D},\mf c_{E}]^{(D,E)}=1\end{equation}
d'indices $E\in\comp(\mc E)$ et  $D\in\comp(\mc D)$.
\end{prop}
\noindent La preuve de cette proposition se fait en appliquant le th\'{e}or\`{e}me classique de Seifert-Van Kampen de fa\c{c}on r\'{e}currente, voir par exemple \cite{Mumford, Dimca, M-M}.\\

Nous \'{e}crirons les \'{e}l\'{e}ments de $\Gamma$ en notation multiplicative et ses classes dans $\Gamma/[\Gamma,\Gamma]\cong H_{1}(\mc T_{\eta}^*;\mb Z)$ en notation additive, mais en conservant les m\^{e}mes noms.

\begin{cor}\label{homologie}
Le groupe d'homologie $H_{1}(\mc T_{\eta}^*;\mb Z)$ est un groupe ab\'{e}lien libre de rang $r:=\#\comp(S)$, engendr\'{e} par les classes $\mf c_{\mc S_{j}}$ associ\'{e}s aux composantes irr\'{e}ductibles $\mc S_{1},\ldots,\mc S_{r}$ de $\mc S$. De plus, si l'on note $\{E_1,\ldots,E_n\}$ les composantes de $\mc E$ et respectivement
$\mf c_{\mc E}$ et $\mf c_{\mc S}$ les vecteurs transpos\'{e}s de $(\mf c_{E_{1}},\ldots,\mf c_{E_{n}})$ 
et de $(\mf c_{\mc S_{1}},\ldots,\mf c_{\mc S_{r}})$, alors
\begin{equation}\label{h1b}
\mf c_{\mc E}=-(\mc E,\mc E)^{-1}(\mc E,\mc S)\mf c_{\mc S},
\end{equation}
o\`{u} $(\mc E,\mc E)$ et $(\mc E,\mc S)$ d\'{e}signent les matrices ayant pour coordonn\'{e}es les nombres d'intersection $(E_{i},E_{j})$ et $(E_{i},\mc S_{k})$ respectivement. Enfin la composante $(i,k)$ de la matrice $-(\mc E,\mc E)^{-1}\cdot(\mc E,\mc S)$ est \'{e}gale \`{a} l'ordre d'annulation $\nu_{\ssstyle E_i}(f_k\circ E)$ de $f_{k}\circ E$ sur $E_{i}$, o\`{u} $f_{k}$ d\'{e}signe une \'{e}quation r\'{e}duite de $S_{k}$.
\end{cor}

\begin{dem}
Il suffit d'\'{e}crire matriciellement les \'{e}quations
\begin{equation}\label{h1}
0=\sum\limits_{D\in\comp(\mc D)}(E_{i},D)\mf c_{D}=\sum\limits_{j=1}^{n}(E_{i},E_{j})\mf c_{E_{j}}+\sum\limits_{k=1}^{r}(E_{i},\mc S_{k})\mf c_{\mc S_{k}}
\end{equation}
qui se d\'{e}duisent des relations (\ref{relgamma}),
puis d'exprimer les $\mf c_{E_{i}}$ en fonction des $\mf c_{\mc S_{k}}$ gr\^{a}ce au fait bien connu que $\det(\mc E,\mc E)$ est \'{e}gal \`{a} $\pm 1$. Finalement,
\begin{eqnarray}
\nu_{\ssstyle E_i}(f_k\circ E)&=&\frac{1}{2i\pi}\int_{\mf c_{E_{i}}}E^*\left(\frac{df_{k}}{f_{k}}\right)=\frac{1}{2i\pi}\int_{-\sum\limits_{\ell=1}^{r}\big((\mc E,\mc E)^{-1}(\mc E,\mc S )\big)_{i\ell}E(\mf c_{\mc S_\ell})}\frac{df_{k}}{f_{k}}\nonumber\\
&=&-\big((\mc E,\mc E)^{-1}\cdot(\mc E,\mc S )\big)_{ik}\,,\label{ord}
\end{eqnarray}
car $\frac{1}{2i\pi}\int_{E(\mf c_{\mc S_\ell})}\frac{df_{k}}{f_{k}}=\delta_{\ell k}$.
\end{dem}

\subsection{D\'{e}composition de JSJ}\label{3.2}
Ce qui suit est bien connu des sp\'{e}cialistes de la topologie des $3$-vari\'{e}t\'{e}s. Il s'agit des applications aux singularit\'{e}s de courbes, de la classification des vari\'{e}t\'{e}s de dimension 3 due \`{a} Waldhausen \cite{Waldhausen}, Jaco-Shalen \cite{JacoShalen} et Johannson \cite{Johannson}. Cette \'{e}tude est effectu\'{e}e par Michel-Weber \cite{MW} et par Neumann \cite{NeumannTAMS,Neumann}, via des techniques de plombage. Dans ce paragraphe nous  pr\'{e}cisons et adaptons ces m\'{e}thodes, afin de mettre en \'{e}vidence les propri\'{e}t\'{e}s  qui nous seront utiles au chapitre suivant.  Pour les \'{e}nonc\'{e}s pr\'{e}cis des th\'{e}or\`{e}mes utilis\'{e}s, nous nous r\'{e}f\'{e}rons \`{a}  \cite{LMW} dont nous adoptons le vocabulaire; le lecteur pourra aussi se r\'{e}f\'{e}rer \`{a} la monographie \cite{Wall} de CTC Wall,  ainsi qu'\`{a} l'article   \cite{NeumannSwarup} de Neumann-Swarup.
\\

Avec les notations (\ref{KD}) et (\ref{XK}), d\'{e}finissons pour tout point singulier $s$ et toute composante $D$ de $\mc D$, les sous-vari\'{e}t\'{e}s \`{a} bord de $\mc M_\eta$ suivantes :
\begin{equation}\label{blocselem}
\mc M_{s}:=
\mc M_{\eta}\cap\Omega_{s}\quad {\rm et} \quad\mc M_D:=\mc M_\eta(K_{D})\,;
\end{equation}
nous les appelons ici \emph{blocs \'{e}l\'{e}mentaires} de $\mc M_\eta$. La d\'{e}composition de $\mc M_\eta$ en ``blocs de Jaco-Shalen-Johannson'' (JSJ) et en ``tores \'{e}paissis'', que nous allons maintenant d\'{e}finir, sera obtenue en agr\'{e}geant de tels blocs.\\

%
Notons $\mf R\subset\comp(\mc D)$ l'ensemble des composantes de $\mc D$ de valence $\ge 3$. Celles-ci correspondent aux  les sommets de rupture de l'arbre dual $\mb A$.
Pr\'{e}cisons la notion de  \emph{cha\^{\i}ne de composantes} reliant  deux \'{e}l\'{e}ments $D'$ et $D''\in \mf R$ : c'est une collection  finie de composantes
\begin{equation}\label{notachaine1}
 \mc C:=\{D_0,\ldots , D_{{l_{\mc C}}+1}\}\,,\quad l_{\mc C}\ge 0,\quad D_0 =D'\,,\quad D_{{l_{\mc C}}+1}=D''\,,
\end{equation}
telle que
\begin{equation}\label{notachaine2}
v(D_1)=\cdots =v(D_{l_{\mc C}})=2\quad \hbox{et}\quad D_j\cap D_{j+1}\neq \emptyset\,,\quad j=0,\ldots ,{l_{\mc C}}\,.
\end{equation}
D\'{e}signons par $\mf C$ l'ensemble de cha\^{\i}nes de composantes de $\mc D$ reliant des composantes de $\mf R$.
Rappelons qu'on appelle \emph{branche morte de $\mc E$ adjacente \`{a} $D\in \mf R$}, toute suite finie  $\mc C :=\{D_0,\ldots , D_{l_{\mc C}}\}$, $l_{\mc C}\geq 1$, de composantes de $\mc E$, telle que 
 \begin{equation}\label{relbrmorte}
D_0=D\,,\quad v(D_j)=2 \,,\quad v(D_{l_{\mc C}})=1\,,\quad D_k\cap D_{k+1}\neq\emptyset\,,
\end{equation}
pour $1\leq j\leq {l_{\mc C}}$ et $0\leq k\leq l_{\mc C}-1$.
La composante $D_{l_{\mc C}}$ est appel\'{e}e \emph{composante d'extr\'{e}mit\'{e} de} $\mc C$ et le point d'intersection de $D_0$ avec $D_1$, \emph{point d'attache de} $\mc C$. Nous d\'{e}signons par $\mf M$ l'ensemble de branches mortes de $\mc E$.\\

Soit $\mc C=\{D_{0},\ldots,D_{\ell_{\mc C}+1}\}\in\mf C$. 
%
D\'{e}signons par $s_j$ le point d'intersection de $D_{j-1}$ et $D_{j}$, $j=0,\ldots ,l_{\mc C}$. Classiquement, pour $\eta >0$ assez petit, ce que nous supposons, $\mc M_{s_j}$ est un \emph{tore \'{e}paissi}, i.e. $\mc M_{s_j}$ est hom\'{e}omorphe au produit   du tore standard  $\mb T:=\partial \mb D_1\times\partial \mb D_1$ avec un intervalle compact. Chaque $\mc M_{D_j}$, $j=1,\ldots ,{l_{\mc C}}$,  est aussi  un tore \'{e}paissi et l'on obtient par recollement une 3-vari\'{e}t\'{e} \`{a} bord $\mc M_{\mc C}$, munie d'un hom\'{e}omorphisme :
\begin{equation}\label{torepaissi}
\breve\sigma _{{}_\mc C} :
\mc M_{\mc C}:=\bigcup _{j=1}^{l_{\mc C}}\mc M_{D_j}\cup\bigcup_{j=0}^{{l_{\mc C}}}\mc M_{s_j} \iso \mb T\times[-1,1]\,.
\end{equation}
Cette structure produit s'\'{e}tend sur un voisinage   du bord de $\mc M_{\mc C}$ sur une 3-vari\'{e}t\'{e} \`{a} bord $\wt{ \mc M}_{\mc C}$, munie d'un hom\'{e}omorphisme
\begin{equation}\label{voistoreepaissi}
\sigma _{{}_\mc C} : \wt{\mc M}_{\mc C}\iso \mb T\times [-1-\epsilon, 1+\epsilon]\,,\ \sigma _{{}_\mc C}^{-1}(\mb T\times[-1, 1]) = \mc M_{\mc C}\,,\  \sigma _{\mc C|\mc M_{\mc C}}=\breve{\sigma }_{\mc C}\,.
\end{equation}


Consid\'{e}rons le 2-tore $\mb T_{\mc C}:= \sigma _{{}_{\mc C}}^{-1}(\mb T\times\{0\})$. L'adh\'{e}rence $B$ de toute composante connexe  de $\mc M_\eta\setminus (\cup_{\mc C\in \mf C}\mb T_{\mc C})$ contient un unique bloc \'{e}l\'{e}mentaire $\mc M_D$, $D\in \mf R$. Nous disons que $B$ est le \emph{bloc de JSJ de $\mc M_\eta$ associ\'{e} \`{a} $D$} et nous le notons $B_D$. Nous d\'{e}signons  par $B_D^\flat$ la composante connexe  de la fermeture de  $\mc M_\eta\setminus \cup_{\mc C\in \mf C}\mc M_{\mc C}$, contenue dans $B_D$.\\

Consid\'{e}rons  une branche morte $\mc C=\{D_0,\ldots , D_{l_{\mc C}}\}\in\mf M$ de $\mc E$. 
Nous notons encore :
\begin{equation}\label{voisbrmorte}
\mc M_{\mc C}:=\bigcup _{j=1}^{l_{\mc C}}\mc M_{D_j}\cup\bigcup_{j=0}^{l_{\mc C}-1}\mc M_{s_j}\,,\quad \hbox{\rm o\`{u}}\quad \{s_j\}:=D_j\cap D_{j+1}\,.
\end{equation}
Ainsi, pour $D\in \mf R$, $B_D^\flat$ est l'union de $\mc M_D$ et des vari\'{e}t\'{e}s $\mc M_{\mc C}$, pour toutes les branches mortes $\mc C$ s'attachant en un point de $D$.
Remarquons que si $\mc C\in\mf M$ alors  ${\mc M_{\mc C}}$ est hom\'{e}omorphe \`{a} un tore solide $\mb D\times \mb S^{1}$; notons aussi que le compl\'{e}mentaire $\mc M_{\mc C}^{\circ}$ d'une fibre de Hopf (non contenue dans $D_{l_{\mc C}-1}$) du diviseur $D_{l_{\mc C}}$ de valence $1$ dans ${\mc M_{\mc C}}$ a le type d'homotopie du tore $\mb S^{1}\times\mb S^{1}$.

\begin{defin}\label{meridienchaine}
Pour $\mc C\in\mf C\cup\mf M$, posons $H_{1}^{\mc C}=H_{1}(\mc M_{\mc C},\mb Z)$, si $\mc C\in\mf C$ et $H_{1}^{\mc C}=H_{1}(\mc M_{\mc C}^{\circ},\mb Z)$, si $\mc C\in\mf M$.
Pour chaque $D_{j}\in\mc C$, la classe $\mf c_{j}$ d'une fibre de $\rho_{D_{j}}$ restreinte \`{a} $\mc M_{D_{j}}$ et orient\'{e}e comme bord d'une courbe holomorphe de $\mc T_{\eta}$, sera appel\'{e}e  \emph{m\'{e}ridien associ\'{e} \`{a} $D_{j}$}.
Si $\mc C\in \mf M$, nous d\'{e}signons par $\mf c_{l_{\mc C}+1}$ et nous appelons  \emph{m\'{e}ridien exceptionnel de $\mc C$}, le  g\'{e}n\'{e}rateur du noyau du morphisme $H_{1}(\mc M^{\circ}_{\mc C},\mb Z)\to H_{1}(\mc M_{\mc C},\mb Z)$
induit par l'inclusion, dont l' orientation et induite par le bord d'une courbe holomorphe de $\mc T_{\eta}$.
\end{defin}

\begin{prop}\label{det}
Si $\mc C\in\mf C\cup\mf M$ alors :
\begin{enumerate}[(i)]
\item
$H_{1}^{\mc C}$ est un groupe ab\'{e}lien libre de rang $2$ qui admet comme syst\`{e}me de g\'{e}n\'{e}rateurs $\mf c_{0},\ldots,\mf c_{l_{\mc C}+1}$, avec les relations\footnote{$e_j$ est aussi \'{e}gal \`{a} la classe de Chern du fibr\'{e} normal de $D_j$ dans $\mc B$, int\'{e}gr\'{e}e sur la classe fondamentale.}
\begin{equation}\label{formindice}
    \mf c_{j-1} + e_j \mf c_j + \mf c_{j+1} = 0\,,\quad e_{j}=(D_{j},D_{j})\,\quad j=1,\ldots, {l_{\mc C}}\,;
\end{equation}
\item
 pour tout $j=0,\ldots,l_{\mc C}$, les \'{e}l\'{e}ments $\mf c_{j},\mf c_{j+1}$ forment base de $H_{1}^{\mc C}$; ces bases d\'{e}finissent toutes  la m\^{e}me orientation;
la  $2$-forme $\mb Z$-lin\'{e}aire canonique $\det(\cdot,\cdot)$ sur $H_{1}^{\mc C}$ telle que $\det(\mf c_{j},\mf c_{j+1})=1$ et $\det(\mf c_{j},c_{j})=0$, correspond  \`{a} la forme d'intersection de chaque composante connexe de $\partial \mc M_{\mc C}$, consid\'{e}r\'{e}e comme surface orient\'{e}e;
%
\item on a $\mf c_{0}=a\,\mf c_{l_{\mc C}}+ b\, \mf c_{l_{\mc C}+1}$, avec $a=\pm\det(A)\neq 0$, o\`{u} $A$ d\'{e}signe la matrice de la restriction au diviseur $\bigcup\limits_{j=1}^{l_{\mc C}}D_{j}$, de la forme d'intersection de $\mc E$;
\item  les \'{e}l\'{e}ments $\mf c_{0}\otimes 1$, $\mf c_{l_{\mc C}+1}\otimes 1$ forment une $\mb Q$-base de $H_{1}^{\mc C}\otimes \mb Q$.
\end{enumerate}
\end{prop}

\begin{dem}
L'affirmation (i) est une cons\'{e}quence directe du fait que $H_{1}^{\mc C}$ s'identifie \`{a} l'homologie enti\`{e}re d'un tore et des relations (\ref{h1}). L'assertion
(ii) se d\'{e}duit facilement \`{a} partir des relations (\ref{formindice}),
 voir aussi \cite{Dimca,EN},  qui s'\'{e}crivent sous forme matricielle comme
\begin{equation}\label{relmatricielle}
\underbrace{
    \left[
\begin{array}{cccccc}
e_1 & 1 & 0 &0& \cdots & 0\\
1 & e_2 & 1 &0&\cdots  & 0\\
0& \ddots& \ddots&\ddots& \ddots& \vdots\\
\vdots& \ddots& \ddots & \ddots & \ddots & 0\\
0 &\cdots &0&1&e_{l_{\mc C}-1} &1\\
0 & \cdots&\cdots & 0 & 1 & e_{l_{\mc C}}
\end{array}\right]}_{ A}
\left[\begin{array}{c}
\mf c_{1}\\
\mf c_{2}\\
\vdots\\
\vdots\\
0\\
\mf c_{{l_{\mc C}}}
\end{array}\right]
= -
\left[
\begin{array}{c}
\mf c_{0}\\ 0\\ \vdots\\\vdots\\0\\ \mf c_{{l_{\mc C}}+1}\end{array}
\right]\;.
\end{equation}
En appliquant la formule de Cramer on voit sans peine que le coefficient $a$ de l'expression $\mf c_{0}=a\mf c_{l_{\mc C}}+b\mf c_{l_{\mc C}+1}$ n'est autre que $a=\pm\det A$, o\`{u} $A$ est la matrice de la restriction au diviseur $\bigcup\limits_{j=1}^{l_{\mc C}}D_{j}\subset\mc E$ de la
forme d'intersection de $\mc E$ qui est d\'{e}finie n\'{e}gative. Ce qui donne (iii), car $\det A\neq 0$.
Finalement, l'affirmation (iv) est une cons\'{e}quence directe de (iii).
\end{dem}

D\'{e}signons par $S_{\mf M}(D)\subset S(D)$  l'ensemble des points d'at\-tache des branches mortes sur $D$ et notons :
\begin{equation}\label{baseseifert}
\wh S(D):=S(D)\setminus S_{\mf M}(D)\,,\quad
   \wh K_D := D\setminus \bigcup\limits_{s\in \wh S(D)}
   \inte D_s\,=\,K_{D}\cup\bigcup\limits_{s\in S_{\mf M}(D)}D_{s}\,.
\end{equation}
\begin{cor}\label{consfibrseifert} Pour chaque composante $D$ de $\mc E$ de valence $\geq 3$, la restriction \`{a} $\mc M_D$ de la fibration $\rho _D$ du syst\`{e}me local $\mc L$, se prolonge en une fibration de Seifert $\wh\rho _D : B_D^\flat\to \wh K_D$, de fibres exceptionnelles   $\wh\rho _D^{-1}(s)$, $s\in S_{\mf M}(D)$. De plus  $\wh\rho _D^{-1}(s)$ est l'intersection de $B_D^\flat$ avec une fibre de la fibration de Hopf de base la composante d'extr\'{e}mit\'{e} de la branche morte qui s'attache \`{a} $D$ au  point $s$.
\end{cor}

\begin{dem}
Donnons-nous $\mf m:=[\partial\mb D\times\{1\}]$ un m\'{e}ridien  et $\mf p:=[\{1\}\times\mb S^{1}]$ un parall\`{e}le
dans $H_{1}(\mb D^*\times \mb S^{1},\mb Z)$. Il est bien connu qu'une courbe de $\mb D^*\times \mb S^{1}$
de classe d'homologie enti\`{e}re  $a\mf  m+b\mf p$ est une fibre d'une fibration de Seifert de $\mb D\times\mb S^{1}$, si et seulement si $a\neq 0$. On conclut en appliquant la partie (iii) de la proposition pr\'{e}c\'{e}dente et en remarquant que $\mf m=\mf c_{l_{\mc C}+1}$.
\end{dem}

\begin{obs} La structure produit des tores \'{e}paissis $\mc M_{\mc C}$, $\mc C\in \mf C$, permet d'\'{e}tendre facilement  $\wh\rho _D$ en une fibration de Seifert
\begin{equation}\label{seifertetendue}
\wh\rho _D^{\rm ext} : B_D\to \wh K_D^{\rm ext}\,,\quad \quad\wh K_D^{\rm ext}:=D\setminus \cup_{s\in \wh S(D)}\check{D}_s\,,\quad \wh\rho _{D|B_D^\flat}^{\rm ext}=\wh \rho _D\,,
\end{equation}
dont les fibres  sont contenues dans les fibres de $\sigma _{{}_{\mc C}}$, $\check{D}_s$ d\'{e}signant ici un disque conforme ferm\'{e} de centre $s$ contenu dans $\inte D_s$.
\end{obs}
Ainsi chaque tore $\mb T_{\mc C}$,   $\mc C\in\mf C$, qui est l'intersections de deux blocs de JSJ, est muni des deux fibrations en cercles, obtenues en restreignant \`{a} $\mb T_{\mc C}$ les fibrations de Seifert de chaque bloc adjacent. Les classes d'homologie des fibres correspondantes \`{a} ces deux fibrations sont $\mf c_{0}$ et $\mf c_{l_{\mc C}+1}$, qu'on peut consid\'{e}rer   comme des \'{e}l\'{e}ments de $H_1(\mb T_{\mc C},\mb Z)$ puisque l'inclusion $\mb T_{\mc C}\subset \mc M_{\mc C}$ est un isomorphisme en homologie.

\begin{obs}\label{incompseifert} Soit $\mc C\in\mf C$. Pour $j\in\{0,l_{\mc C}+1\}$, en utilisant que $v(D_{j})\ge 3$ on voit ais\'{e}ment \cite{M-M} que
$\mb T_{\mc C}$ est incompressible dans $B_D$, ce qui donne le monomorphisme
%
$H_{1}(\mb T_{\mc C},\mb Z)\hookrightarrow H_{1}(B_{D_{j}},\mb Z)$.
Ainsi $\mf c_{0}$ et $\mf c_{l_{\mc C}+1}$ sont aussi ind\'{e}pendantes dans $H_{1}(B_{D_{j}},\mb Z)$  et donc les fibrations de Seifert de $B_{D_{0}}$ et $B_{D_{l_{\mc C }+1}}$ ne sont pas compatibles. En utilisant les relations (\ref{h1b}), il est facile de voir que l'image de $H_{1}(\mb T_{\mc C},\mb Z)$ dans $H_{1}(\mc M_{\eta},\mb Z)$ est diff\'{e}rente des images de l'homologie des tores  bordant  $\mc M_{\eta}$.
\end{obs}

Les hypoth\`{e}ses du th\'{e}or\`{e}me (1.2.3) de \cite{LMW} \'{e}tant satisfaites, il vient :
\begin{cor}
La famille  $(\mb T_{\mc C})_{\mc C\in \mf C}$ est  une famille caract\'{e}ristique\footnote{i.e. une famille minimale de tores telle l'adh\'{e}rence de chaque composante connexe du compl\'{e}mentaire est une  vari\'{e}t\'{e} de Seifert ou atoro\"{\i}dale, cf. \cite{LMW} p. 144.} de tores essentiels\footnote{i.e. incompressible dans $\mc M_{\eta}$ et non-isotope \`{a} une composante de $\partial \mc M_\eta$.} de la 3-vari\'{e}t\'{e} $\mc M_{\eta}$ et d\'{e}termine sa d\'{e}composition de JSJ, qui est constitu\'{e}e de blocs de type Seifert.
\end{cor}

\begin{obs}\label{grapheJSJ}
Les sommets de l'arbre de la d\'{e}composition de JSJ de $\mc M_{\eta}$ (correspondants aux blocs Seifert $B_{D}$) sont en correspondance bijective avec les composantes $D\in\mf R$ et ses ar\^{e}tes (joignant deux sommets correspondants a deux blocs Seifert adjacents) sont en correspondance bijective avec les cha\^{\i}nes $\mc C\in\mf C$.
\end{obs}

\subsection{Structures p\'{e}riph\'{e}riques et isomor\-phis\-mes g\'{e}om\'{e}triques}\label{est-per}

Pour chaque composante irr\'{e}ductible $ S_{k}$ de $ S$ consid\'{e}rons un voisinage tubulaire $ W_{k}$ de $S_{k}\cap (\mb B_{r}\setminus\inte{\mb B}_{s})$ avec $0<s<r\ll 1$  tel que les restrictions des fibrations $\rho_{\mc S_{k}}$ et $\rho_{D_{k}}$ \`{a} $\mc W_{k}:=E^{-1}(W_{k})$ soient des fibrations triviales,  $D_{k}\in\comp(\mc D)$ d\'{e}signant la composante d'attache de $\mc S_{k}$.
Le groupe fondamental $\mc P_{k}:=\pi_{1}(W_{k}^*)$ est isomorphe \`{a} $\mb Z\mf m_{k}\oplus\mb Z\mf p_{k}$ o\`{u}
$\mf m_{k}$ et $\mf p_{k}$ sont les bords orient\'{e}s d'une fibre de la restriction \`{a} $\mc W_{k}^*$ de $\rho_{\mc S_{k}}$ et $\rho_{D_{k}}$ respectivement. L'ab\'{e}lianit\'{e} de $\mc P_{k}$ permet de ne pas \`{a} expliciter un point de base dans $W_{k}^*$.
Remarquons que $\mf m_{k}$ est un g\'{e}n\'{e}rateur du noyau du morphisme $\pi_{1}(W_{k}^*)\to\pi_{1}(W_{k})$ induit par l'inclusion.
Soit $s=\mc S_{k}\cap D_{k}\in\sing(\mc D)$ le point d'attache de $\mc S_{k}$. Quitte \`{a} permuter les coordonn\'{e}es $(x_{s},y_{s})$ nous supposons que $x_{s}=0$ est une \'{e}quation de $\mc S_{k}$. Nous choisissons $\varepsilon _1$, $\varepsilon_2>0$ convenables, pour que $\mc W_{k}^*$ se r\'{e}tracte sur le $2$-tore $\{|x_{s}|=\varepsilon_1,\, |y_{s}|=\varepsilon_2\}$.
Les lacets $m$ et $p$ de $\mc W_{k}^*$ d\'{e}finis par $(x_{s},y_{s})\circ m(t)=(\varepsilon_1 e^{2i\pi t},\varepsilon_2)$ et $(x_{s},y_{s})\circ p(t)=(\varepsilon_1 ,\varepsilon_2 e^{2i\pi t})$ sont deux repr\'{e}sentants de $\mf m_{k}$ et $\mf p_{k}$ respectivement.

\begin{defin}
Nous appellerons $\mf m_{k}$ et $\mf p_{k}$ \emph{le m\'{e}ridien} et \emph{le parall\`{e}le} canonique de $S_{k}$.
\end{defin}

\begin{prop}\label{bord}
L'ensemble $W_{k}^*$ est incompressible dans $T_{\eta}^*$, i.e. le morphisme $i_{k}:\mc P_{k}\to\Gamma$ induit par l'inclusion $W_{k}^*\subset T_{\eta}^*$, qui s'explicite par $i_{k}(\mf m_{k})=\mf c_{S_{k}}$, $i_{k}(\mf p_{k})=\mf c_{D_{k}}$, est injectif.
\end{prop}

\begin{dem}
Ceci peut se d\'{e}montrer
\footnote{Lorsque $S$ n'est pas irr\'{e}ductible, i.e. $r>1$, on peut raisonner directement en homologie, car alors $\pi_{1}(W_{k}^*)\cong H_{1}(W_{k}^*;\mb Z)\cong\mb Z^{2}\hookrightarrow\mb Z^{r}\cong H_{1}(T_{\eta}^*;\mb Z)$. Ce derni\`{e}re inclusion vient de (\ref{ord}) car $f_{\ell}\circ E$ s'annule sur $D_{k}$, pour tout $\ell=1,\ldots,r$.}
directement, par utilisation r\'{e}p\'{e}t\'{e}e du th\'{e}or\`{e}me de Van Kampen, comme nous l'avons d\'{e}j\`{a} fait en la construction d'un voisinage adapt\'{e} de $\mc D$ par ``assemblage bord \`{a} bord'' dans \cite{M-M}.
Nous allons pr\'{e}senter une autre preuve bas\'{e}e sur 
l'incompressibilit\'{e} dans $T_{\eta}^*$ de la fibre de Milnor\footnote{Elle r\'{e}sulte trivialement de la suite exacte d'homotopie de la fibration de Milnor.} $F$ d'une \'{e}quation r\'{e}duite $f$ de $S$. Notons par $i_{CW}, i_{WT}, i_{CF}, i_{FT}$ respectivement les morphismes au niveau des groupes fondamentaux, induits par les inclusions $F\cap W_{k}^*\subset W_{k}^*, W_{k}^* \subset T_{\eta}^*, F\cap W_{k}^*\subset F, F\subset T_{\eta^*}$.
Remarquons que $\pi_{1}(F)$ est le noyau du morphisme $f_{*}:\Gamma\to\mb Z$ qui envoie $\mf c_{D}$ sur la multiplicit\'{e}  $\nu_{D}(f\circ E)$ de $f\circ E$ le long de $D$. Notons  $\nu_{k}:=\nu_{D_k}(f\circ E)$.
Comme $f\circ E =x_{s}y_{s}^{\nu_{k}}$,
on a l'isomorphisme  $\pi_{1}(F\cap W_{k}^*)\cong\mb Z\mf b_{k}$, o\`{u} $\mf b_{k}$ d\'{e}signe la composante du bord (orient\'{e}) de $F$ contenue dans $W_{k}^*$  et $i_{CW}(\mf b_{k})=\mf p_{k}-\nu_{k}\mf m_{k}$.
D'autre part, si $\mf k=\alpha \mf p_{k}+\beta \mf m_{k}\in\pi_{1}(W_{k}^*)$ appartient au noyau de $i_{WT}$, alors $f_{*}(i_{WT}(\alpha \mf p_{k}+\beta \mf m_{k}))=\alpha\nu_{k}+\beta=0$; d'o\`{u} $\mf k=i_{CW}(\alpha \mf b_{k})$. Comme $i_{WT}\circ i_{CW}=i_{FT}\circ i_{CF}$, $i_{CF}$ et $i_{FT}$ sont injectives,  $\alpha=0$ et $i_{WT}$ est donc aussi injective.
\end{dem}

Dor\'{e}navant nous identifierons $\mc P_{k}$ a son image dans $\Gamma$, en prenant le point de base dans $W_{k}^*$. Si nous avons besoin de consid\'{e}rer plus d'un sous-groupe $\mc P_{k}$ \`{a} la fois, il nous faudra alors de consid\'{e}rer la famille de tous les conjugu\'{e}s de$\mc P_{k}$ dans $\Gamma$. Le r\'{e}sultat suivant pr\'{e}cise cette situation.

\begin{prop}
Le normalisateur de $\mc P_{k}$ dans $\Gamma$ est \'{e}gal \`{a} $\mc P_{k}$, i.e. si $\zeta\in\Gamma$ et $\zeta\mc P_{k}\zeta^{-1}\subset\mc P_{k}$ alors $\zeta\in\mc P_{k}$. En particulier, la d\'{e}composition $\mc P_{k}=\mb Z\,\mf m_{k}\oplus\mb Z\,\mf p_{k}$ est intrins\`{e}que
\footnote{i.e. la d\'{e}composition   $P=\mb Z\mf m_{\ssstyle P}\oplus\mb Z\mf p_{\ssstyle P}$ de tout sous-groupe conjugu\'{e} $P:=\zeta \mc P_k\zeta ^{-1}$ donn\'{e}e par $\mf m_{\ssstyle P}:=\zeta \mf m_k\zeta ^{-1}$ et $\mf p_{\ssstyle P}:=\zeta \mf p_k\zeta ^{-1}$, ne d\'{e}pend pas de $\zeta $.}
dans $\Gamma$.
\end{prop}

\begin{dem}
La preuve de la proposition pr\'{e}c\'{e}dente montre que $\pi_{1}(F)\cap\mc P_{k}=\pi_{1}(F\cap W_{k}^*)=\mb Z\mf b_{k}$.
Soit $\zeta':=\zeta\mf m_{k}^{\ell}$ avec $\ell:=f_{*}(\zeta)=\frac{1}{2i\pi}\int_{\zeta}E^*\left(\frac{df}{f}\right)$. Comme $f_{*}(\mf m_{k})=1$ il r\'{e}sulte que $f_{*}(\zeta')=0$ et donc $\zeta'\in\pi_{1}(F)$. Ainsi, $\zeta'\mf b_{k}\zeta'^{-1}\in\pi_{1}(F)\cap \zeta'\mc P_{k}\zeta'^{-1}=\pi_{1}(F)\cap\mc P_{k}=\mb Z\mf b_{k}$. D'o\`{u}  $\zeta'\mf b_{k}\zeta'^{-1}=\mf b_{k}^{n}$ pour un certain $n\in\mb Z$. En passant \`{a} l'homologie la derni\`{e}re \'{e}galit\'{e} on obtient que $n=1$ et donc $[\zeta',\mf b_{k}]=1$, que nous pouvons interpr\'{e}ter comme une relation dans le groupe libre $\pi_{1}(F)$. Comme le sous-groupe $\langle\zeta',\mf b_{k}\rangle$ de $\pi_{1}(F)$ est aussi libre d'apr\`{e}s le th\'{e}or\`{e}me classique de Schreier, nous en d\'{e}duisons qu'il est monog\`{e}ne : $\langle \zeta',\mf b_{k}\rangle=\langle \theta\rangle$ pour un certain $\theta\in\pi_{1}(F)=\langle u_{1},v_{1},\ldots,u_{g},v_{g},b_{1},\ldots,b_{r}| \prod\limits_{i=1}^{g}[u_{i},v_{i}]\prod\limits_{j=1}^{r}b_{j}=1\rangle$, o\`{u} $b_{j}\subset\partial F$. Nous pouvons supposer que $\mf b_{k}=b_{1}$. Il suffit de prouver que $b_{1}$ n'est pas une puissance non-triviale dans $\pi_{1}(F)$ car dans ce cas $\zeta'\in\langle\theta\rangle=\langle\mf b_{k}\rangle\subset\mc P_{k}$.
Si $r>1$ alors $b_{1}$ fait parti du syst\`{e}me libre de g\'{e}n\'{e}rateurs $u_{1},v_{1},\ldots,u_{g},v_{g},b_{1},\ldots,b_{r-1}$ de $\pi_{1}(F)$; il n'est donc pas une puissance non-triviale.
Si $r=1$ alors $b_{1}^{-1}=\prod\limits_{i=1}^g[u_{i},v_{i}]$ est un mot cycliquement r\'{e}duit dans le groupe libre $\pi_{1}(F)=\langle u_{1},v_{1},\ldots,u_{g},v_{g}|\,\rangle$; on voit  encore facilement, qu'il ne peut \^{e}tre  une puissance non-triviale.
\end{dem}

\begin{teo}\label{pk}
Soit $U$ un voisinage ouvert de $0$ dans $\mb B$ et $h$ un hom\'{e}o\-mor\-phisme\footnote{pr\'{e}servant  comme toujours les orientations.} de $U$ sur un voisinage $U'$ de $0$ dans $\mb B'$, tel que $h(S\cap U)=S'\cap U'$. Supposons que l'inclusion $U\subset \mb B$ induit un isomorphisme $\pi_{1}(U^*)\cong\Gamma$.
Alors, pour toute composante $S_{k}$ de $S$ l'isomorphisme $h_{*}:\Gamma\to\Gamma'$ induit par $h$ envoie $\mc P_{k}$ sur le sous-groupe $\mc P_{k}'$ associ\'{e} \`{a} la composante $S_{k}'=h(S_{k}\cap U)$ de $S'\cap U'$ et transforme m\'{e}ridien en m\'{e}ridien : $h_{*}(\mf m_{k})=\mf m_{k}'$.
\end{teo}

\begin{dem}
Consid\'{e}rons des voisinages tubulaires $W_{k}$ de $S_{k}\cap (\mb B_{r}\setminus\inte{\mb B}_{s})$ et $W_{k}''\subset W_{k}'$ de $S'_{k}\cap (\mb B'_{r'}\setminus\inte{\mb B'}_{s'})$ contenus dans $U$ et $U'$ respectivement tels que $W_{k}''\subset h(W_{k})\subset W_{k}'$ et $\mc P_{k}=\pi_{1}(W_{k}^*)$ et $\pi_{1}(W_{k}''^*)=\pi_{1}(W_{k}'^*)=\mc P_{k}'$ via l'inclusion $W_{k}''^*\subset W_{k}'^*$.
Ainsi $h_{*}(\mc P_{k})\subset\mc P_{k}'$ et le compos\'{e} $\mc P_{k}'\to h_{*}(\mc P_{k})\to \mc P_{k}'$ est un isomorphisme. Ainsi $h_{*}(\mc P_{k})=\mc P_{k}'$ et la restriction de $h_{*}$ \`{a} $\mc P_{k}\cong\mb Z^{2}$ est surjective sur $\mc P_{k}'\cong\mb Z^{2}$. Comme tout morphisme surjectif de $\mb Z^{2}$ sur lui-m\^{e}me est aussi injectif,  $h_{*}:\mc P_{k}\to\mc P_{k}'$ est un isomorphisme. De m\^{e}me $h_{*}:\pi_{1}(W_{k})\to\pi_{1}(W_{k'}) $ est aussi un isomorphisme. Ainsi $h_{*}$ conjugue les noyaux des morphismes induits par les inclusions $W_{k}^*\subset W_{k}$ et $W_{k}'^*\subset W_{k}'$ qui sont engendr\'{e}s par $\mf m_{k}$ et $\mf m_{k}'$ respectivement. On conclut que $h_{*}(\mf m_{k})=\mf m_{k}'^{\pm 1}$; mais l'exposant est \'{e}gal \`{a} $+1$, car $h$ pr\'{e}serve les orientations.
\end{dem}

Dans l'\'{e}nonc\'{e} du th\'{e}or\`{e}me pr\'{e}c\'{e}dent, $k$ \'{e}tant donn\'{e}, nous avions arbitrairement choisi dans $W_{k}$ et $W_{k}''$, les  points de base pour $\Gamma$ et $\Gamma'$ respectivement. Mais nous aimerions disposer d'une notion ind\'{e}pendante de ces choix, intrins\`{e}que \`{a} $h_{*}$.  Pour cela, revenons \`{a}  la d\'{e}finition d'\'{e}quivalence fondamentale introduite en~\ref{equivfond} et notons que  l'ambigu\"{\i}t\'{e}  de l'action $h_{*}:\Gamma\to\Gamma'$ est contr\^{o}l\'{e}e par la composition \`{a} droite et/ou \`{a} gauche de $h_\ast$ par des automorphismes int\'{e}rieurs. Cela nous conduit \`{a} introduire  la notion d'\emph{isomorphisme ext\'{e}rieur}, comme une classe d'\'{e}quivalence d'isomorphisme $\Gamma\to\Gamma'$ modulo composition par des automorphismes int\'{e}rieurs. Nous pouvons alors d\'{e}finir

\begin{defin}\label{isomgeom}
Nous disons qu'un isomorphisme ext\'{e}rieur $\varphi:\Gamma\to\Gamma'$ \emph{pr\'{e}serve les structures p\'{e}riph\'{e}riques} s'il envoie tous les sous-groupes conjugu\'{e}s des $\mc P_{k}$ sur des sous-groupes conjugu\'{e}s des $\mc P_{k'}'$.
L'isomorphisme $\varphi$ est dit \emph{g\'{e}om\'{e}trique} si de plus il envoie tous les conjugu\'{e}s des m\'{e}ridiens $\mf m_{k}$ sur des conjugu\'{e}s des m\'{e}ridiens $\mf m_{k'}'$.
\end{defin}

\begin{obs}
Le th\'{e}or\`{e}me \ref{pk} affirme que si $h:(U,S)\to (U',S')$ est un germe d'hom\'{e}omorphisme alors $h_{*}:\Gamma\to\Gamma'$ est un isomorphisme g\'{e}om\'{e}trique.
La premi\`{e}re moiti\'{e} de la preuve de \ref{pk} implique que si $h:U^*\to U'^*$ est un hom\'{e}omorphisme alors $h_{*}:\Gamma\to\Gamma'$ pr\'{e}serve la structure p\'{e}riph\'{e}rique; cependant il peut ne pas \^{e}tre g\'{e}om\'{e}trique, comme montre l'exemple suivant : $U=U'=\mb C^{2}$, $S=S'=\{xy=0\}$ et $h:\mb C^*\times\mb C^*\to\mb C^*\times\mb C^*$ est d\'{e}fini par $h(x,y)=(xy,y)$.
\end{obs}

\noindent Rappelons ici un important r\'{e}sultat de F. Waldhausen \cite[Corollary 6.5]{Waldhausen} :

\begin{teo}
Soient $M$ et $M'$ des vari\'{e}t\'{e}s de dimension trois, irr\'{e}duc\-ti\-bles,  \`{a} bord incompressible et soit $\varphi:\pi_{1}(M)\to\pi_{1}(M')$ un isomorphisme pr\'{e}servant la structure p\'{e}riph\'{e}rique, i.e. pour toute composante connexe $F$ de $\partial M$, il existe une composante connexe $F'$ de $\partial M'$, telle que $\varphi(\pi_{1}(F))$soit conjugu\'{e} \`{a} $\pi_{1}(F')$. Alors il existe un hom\'{e}omorphisme $\phi:M\to M'$ induisant $\varphi $ en homotopie, i.e. $\varphi=\phi_{*}$.
\end{teo}

\begin{cor}\label{invgeo}
Si $\varphi:\Gamma\to\Gamma'$ est un isomorphisme qui pr\'{e}serve la structure p\'{e}riph\'{e}rique, alors il existe un hom\'{e}omorphisme $h:T_{\eta}^*\to T_{\eta}'^*$, tel que $h_{*}=\varphi : \pi _1(T_{\eta}^*)\to \pi _1(T_{\eta}'^*)$.
Si de plus $\varphi$ est g\'{e}om\'{e}trique, alors $h$ s'\'{e}tend en un hom\'{e}omorphisme de $T_{\eta}^*$ sur $T_{\eta}'^*$, tel que  $h(S)=S'$. Ainsi, tout isomorphisme g\'{e}om\'{e}trique est induit par un (unique) marquage.
\end{cor}

\begin{dem}
Nous  pouvons appliquer le th\'{e}or\`{e}me de Waldhausen \`{a} l'isomorphisme $\varphi:\Gamma\cong\pi_{1}(M_{\eta})\to\pi_{1}(M_{\eta}')\cong\Gamma'$, car $M_{\eta}$ et $M_{\eta'}$ sont irr\'{e}ductibles, d'apr\`{e}s la remarque \ref{EilenMacLane} et  \`{a} bord incompressible, gr\^{a}ce \`{a} la proposition~\ref{bord}.
Il existe donc un hom\'{e}omorphisme $\phi:M_{\eta}\to M_{\eta'}$, qui s'\'{e}tend trivialement en un hom\'{e}omorphisme $h:T_{\eta}^*\to T_{\eta'}^*$, via les structures produit $T_{\eta}^*\cong M_{\eta}\times]0,\eta]$ et
 $T_{\eta}'^*\cong M_{\eta}'\times]0,\eta]$ donn\'{e}es par  (\ref{strproduit}).
D'autre part, si $\varphi$ conjugue les m\'{e}ridiens des tores du bord de $M_{\eta}$ et $M_{\eta}'$, alors $\phi$ s'\'{e}tend en un hom\'{e}omorphisme de $T_{\eta}\cap\partial\mb B$ sur $T_{\eta}'\cap\partial\mb B'$. En utilisant la structure conique de $S$ et $S'$, il est facile d'\'{e}tendre $\phi$ en un hom\'{e}omorphisme  entre les paires $h:(T_{\eta},S)\to (T_{\eta}',S')$.
\end{dem}

\section{D\'{e}monstration du th\'{e}or\`{e}me principal}\label{4}

\'{E}tant donn\'{e} l'hom\'{e}omorphisme $h:\mb B_{\varepsilon}\iso h(\mb B_{\varepsilon})\subset\mb B'$ tel que $h(S\cap\mb B_{\varepsilon})=S'\cap h(\mb B_{\varepsilon})$, dans
la premi\`{e}re section de ce chapitre nous construisons une application $\breve h_1$ de $M_\eta$ sur $M'_{\eta'}$, pour $0<\eta \ll \eta'\ll 1$, qui est fondamentalement \'{e}quivalente \`{a} $h$. Gr\^{a}ce aux r\'{e}sultats de Waldhausen, nous modifierons cette application par une homotopie, afin d'obtenir un hom\'{e}omorphisme $h_{2}$ entre les $3$-tubes de Milnor.

Dans la section suivante, en utilisant les r\'{e}sultats classiques  de Jaco-Shalen-Johannson, nous isotopons $h_{2}$ \`{a} un nouvel hom\'{e}omorphisme $h_{3}$ qui pr\'{e}serve des r\'{e}alisations tr\`{e}s pr\'{e}cises de la d\'{e}composition JSJ des $3$-tubes de Milnor.

Ensuite \`{a} la section \ref{4.3}, nous construisons un isomorphisme explicite entre les arbres duaux des d\'{e}singularisations minimales de $S$ et de $S'$.

Celui-ci nous permet, \`{a} la section suivante,
d'\'{e}tendre $h_{3}$ aux $4$-tubes de Milnor. Cette extension \`{a} la dimension quatre se fait faite en quatre \'{e}tapes :
Dans la premi\`{e}re, nous ne nous occupons que des blocs $\mc T_{\eta}(D)$ associ\'{e}s aux composantes de valence $\ge 3$.
Dans la deuxi\`{e}me \'{e}tape, nous traitons le cas des cha\^{\i}nes $\mc C$ de composantes de valence $2$, en utilisant la structure produit des blocs $\mc T_{\eta}(\mc C)$.
\`{A} l'\'{e}tape suivante, nous consid\'{e}rons le cas des branches mortes et celui des transform\'{e}es strictes des s\'{e}paratrices.
Finalement \`{a} la derni\`{e}re \'{e}tape,
nous modifierons l'hom\'{e}omorphisme construit, \`{a} l'aide d'isotopies bien choisies,  afin d'assurer qu'il est fondamentalement \'{e}quivalent \`{a} l'hom\'{e}omorphisme $h$ initial.

\subsection{R\'{e}duction \`{a} la dimension trois}\label{reductrois}
Fixons des r\'{e}els positifs $r<\varepsilon$ et $\varepsilon''<\varepsilon '$, tels que
\begin{equation}\label{inclusoignons}
\mb B'_{\varepsilon''}\varsubsetneq h(\mb B_r)\varsubsetneq\mb B'_{\varepsilon'}\varsubsetneq h(\mb B_\varepsilon)\subset \mb B'{}^\ast\,.
\end{equation}
Munissons le couple $(\mb B, S)$
d'une \emph{structure conique}, c'est \`{a} dire d'un diff\'{e}o\-mor\-phisme $\varphi : \partial \mb B\times [0,1]\to \mb B$ en dehors de l'origine, v\'{e}rifiant : $\varphi(\partial \mb B\times \{r\})= \partial \mb B_r$, pour tout $r\in ]0,1]$,  $\varphi((S\cap \partial\mb  B)\times [0,1])=S $ et $\varphi(m,0)=0$, $\varphi(m,1)=m$, pour tout $m\in \partial \mb B$. Nous disposons aussi d'une structure conique $\varphi' : \partial \mb B'\times[0,1]\to \mb B'$, pour le couple $(\mb B',S')$.
Notons $\varrho_{{}_0} : \mb B\to \mb B_\varepsilon$ la r\'{e}traction par d\'{e}formation qui  correspond, via $\varphi $, \`{a} l'application valant $(m,t)\mapsto (m,\varepsilon)$, pour $\varepsilon\leq t\leq 1$.
Notons aussi $\sigma '_0 : (\mb B'\setminus \{0\})\to \partial\mb B'$ la r\'{e}traction par d\'{e}formation correspondant, via $\varphi'$, \`{a} $(m,s)\mapsto (m,1)$. D\'{e}signons enfin par  $\sigma ': \mb B'\to \mb B'$,  l'application continue   correspondant, via $\varphi '$, \`{a} l'application $(m,t)\mapsto (m,\varsigma(t))$, o\`{u} $\varsigma(t)$ est   affine pour $\varepsilon''\leq t\leq \varepsilon '$ et v\'{e}rifie $\varsigma (t)=t$, pour $ t\leq \varepsilon''$ et $\varsigma(t)=1$, pour $t\geq \varepsilon'$.
Visiblement nous avons:
$$
\varrho^{-1}_{{}_0}(S)=S\,,\quad \sigma '(S')=S'\,,\quad \sigma '{}^{-1}(S')=S'\,,\quad \sigma '{}^{-1}(S')=S'\,.
$$
Notons aussi que $\varrho_{{}_0}$ et $\sigma '$ sont l'identit\'{e} au voisinage de l'origine et que $\sigma '$ co\"{\i}ncide avec $\sigma '_0$ en dehors de $\mb B_{\varepsilon'}'$.
Posons :
\begin{equation}\label{defh1}
h_1:=\sigma '\circ h\circ \varrho_{{}_0} \; :\; \mb B\longrightarrow \mb B'\,.
\end{equation}
Cette application est continue, n\'{e}cessairement surjective et   v\'{e}rifie $h_1(\partial \mb B)=\partial \mb B '$, $h_1(S)=S'$ et $h_1^{-1}(S')=S$. Elle d\'{e}finit ainsi une application de $\mb B^\ast$ dans $\mb B'{}^\ast$. D'autre part $h_1$ co\"{\i}ncide avec $h$ au voisinage de l'origine et donc les restrictions de $h$ et de $h_1$ \`{a} $\mb B^\ast_{\varepsilon}$, que nous d\'{e}signerons encore par $h$ et $h_1$, sont fondamentalement \'{e}quivalentes : $h\asymp h_1$.

Fixons maintenant des 4-tubes de Milnor $T_\eta\subset \mb B$ pour $S$ et $T'_{\eta'}\subset \mb B'$ pour $S'$, tels que $h_1(T_\eta) \subset T'_{\eta'}$. D\'{e}signons par $r: T_\eta^\ast\to M_\eta$ la r\'{e}traction par d\'{e}formation sur le 3-tube de Milnor  (\ref{Meta}),  donn\'{e}e par la structure produit  d\'{e}crite en (\ref{retractT}) et notons $r' : T'^\ast_{\eta'} \to M_{\eta'}':=f'{}^{-1}(\partial\mb D_{\eta'})\cap \mb B'$ la  r\'{e}traction similaire. Posons :
\begin{equation}\label{MetTeta}
    \breve h_1 := r'\circ h_1\circ \iota_{M_\eta} : M_\eta\to M'_{\eta'}
     \,,
\end{equation}
o\`{u} $\iota_{M_\eta} :M_\eta\hookrightarrow T_\eta^\ast$ d\'{e}signe l'application d'inclusion.
\begin{obs}
On a  :
$
h\asymp h_1 \asymp \breve h_1 
 .$
\end{obs}

 Quitte \`{a} multiplier l'\'{e}quation $f'$ par $\frac{\eta}{\eta'}$, nous supposons que $\eta=\eta'$.
Nous identifions d\'{e}sormais $T_\eta$ \`{a} $\mc T_\eta$, $T'_\eta$ \`{a} $\mc T'_\eta$ et nous continuons \`{a} noter $\breve h_1 $  l'application
$E'{}^{-1}\circ \breve h_1\circ E$ d\'{e}finie sur $\mc M_\eta:=E^{-1}(M_\eta)$ et \`{a} valeurs dans $\mc M'_{\eta'}:=E'{}^{-1}(M'_{\eta'})$. Celle-ci satisfait les hypoth\`{e}ses du th\'{e}or\`{e}me (6.1) de Waldhausen  \cite{Waldhausen}, car $\breve h_{1}(\partial \mc M_{\eta})\subset\partial\mc M'_{\eta'}$. Comme $\mc M_\eta$ n'est pas l'espace total d'un fibr\'{e} en droites sur une surface de Riemann close, il existe une homotopie $F :\mc M_\eta \times[0,1] \to \mc M'_{\eta'}$ v\'{e}rifiant $F(\partial \mc M_\eta\times[0,1])\subset \partial\mc M'_{\eta'}$,   $F(\cdot,0)=\breve h_1$ et telle que $F(\cdot,1)$ soit un hom\'{e}omorphisme. Nous posons
\begin{equation}\label{hdeuxM}
    h_2:= F(\cdot, 1) : \mc M_\eta \iso \mc M_\eta'\,.
\end{equation}
\begin{obs} La relation $h_2\asymp \breve  h_1$ est satisfaite.
\end{obs}

\subsection{Construction d'un hom\'{e}omorphisme  JSJ-compatible}
\label{extvaltrois}
Consi\-d\'{e}\-rons
maintenant pour $\mc M_{\eta'}'$, la d\'{e}composition de JSJ    similaire \`{a} celle effectu\'{e}e pour $\mc M_{\eta}$ : nous conservons les notations (\ref{blocselem}) pour les blocs \'{e}l\'{e}men\-tai\-res de $\mc M_{\eta'}'$;
  Nous d\'{e}signons par $\mf R'$ la collection des composantes de $\mc D'$ de valence $\geq 3$ et par $\mf C'$  celle des cha\^{\i}nes de  composantes de $\mc D'$ reliant deux \'{e}l\'{e}ments de $\mf R'$; pour chaque $\mc C'\in \mf C'$, les tores \'{e}paissi $\mc M_{\mc C'}'$ et $\wt{\mc M}_{\mc C'}'$, ainsi que leur structure produit  $\sigma '_{\mc C'} : \wt{\mc M}_{\mc C'}'\iso \mb T\times[-1-\epsilon,+1+\epsilon]$ sont construits comme en (\ref{voistoreepaissi}); le $2$-tore  $\mb T_{\mc C'}'=\sigma_{\mc C'}'^{-1}(\mb T\times\{0\})$ est proprement plong\'{e} dans $\mc M_{\eta'}'$
et les adh\'{e}rences des composantes connexes de $\mc M'_{\eta'}\setminus \cup_{\mc C'\in \mf C'}\mb T_{\mc C'}'$ constituent les \emph{ blocs de JSJ de} $\mc M_{\eta'}'$; chacun d'eux est not\'{e} $B'_{D'}$, car il contient un unique   bloc \'{e}l\'{e}mentaire $\mc M'_{D'}$,  $v(D')\geq 3$; une fibration de Seifert \'{e}tendue $\wh{\rho}'{}^{\rm ext} : B'_{D'}\to \wh K_{D'}'{}^{\rm ext}$, d\'{e}finie comme en (\ref{seifertetendue}), prolonge la fibration de Hopf $\rho '_D : \mc M_{D'}'\to K'_{D'}$; enfin la famille $(\mb T'_{\mc C'})_{\mc C'\in \mf C'}$ est, pour les m\^{e}mes raisons, une famille caract\'{e}ristique  de tores essentiels proprement plong\'{e}s dans $\mc M_{\eta'}'$.

Visiblement $(\mb T_{\mc C})_{\mc C\in \mf C}$ et $(h_2^{-1}(\mb T_{\mc C'}'))_{\mc C'\in \mf C'}$ sont deux familles carac\-t\'{e}\-ris\-tiques de tores essentiels de $\mc M_\eta$. D'apr\`{e}s le th\'{e}or\`{e}me d'unicit\'{e} des familles carac\-t\'{e}\-ris\-tiques, cf. (1.2.6) de \cite{LMW}, il existe une bijection
 \begin{equation}\label{corespcaines}
 \kappa _2 :\mf C\iso \mf C'
\end{equation}
et un automorphisme $\psi$ de $\mc M_{\eta}$ isotope \`{a} l'identit\'{e}, tels que $h_2(\psi(\mb T_{\mc C}))=\mb T'_{\kappa _2(\mc C)}$, pour tout $\mc C\in \mf C$. Posons $\wt h_2:= h_2\circ\psi$, nous avons :
\begin{equation}\label{hdeuxtilde}
  \wt h_2\asymp h_2\asymp h\quad  \hbox{\rm et}\quad \wt h_2(\mb T_{\mc C})=\mb T'_{\kappa _2(\mc C)}\,,\quad\textrm{pour tout}\quad\mc C\in \mf C\,.
\end{equation}
\begin{obs}
Visiblement $\wt h_2$ transforme tout bloc de JSJ de $\mc M_\eta$ en un bloc de JSJ de $\mc M_{\eta'}'$, d\'{e}finissant ainsi une (unique) bijection
$ \kappa_3 : \mf R\iso\mf R'$ telle que $\wt h_{2}(B_{D})=B'_{\kappa_{3}(D)}$.
\end{obs}
\begin{lema}[de l'accord\'{e}on]\label{accordeon} Il existe un hom\'{e}omorphisme $\breve h_2$ isotope \`{a} $\wt h_2$, qui transforme tout tore \'{e}paissi en un tore \'{e}paissi, en respectant la structure produit; pr\'{e}cis\'{e}ment : $\breve h_2(\mc M_{\mc C})=\mc M_{\kappa _2(\mc C)}'$ et $\sigma ' _{\kappa _2(\mc C)}\circ \breve h_2=\sigma _{\mc C}$.
\end{lema}
\begin{dem}
Soit $\mc C\in \mf C$ et $B_D$ une composante de JSJ de $\mc M_\eta$, telle que $\mb T_{\mc C}\subset \partial B_D$. Le tore $\mb T'_{\mc C'}$, $\mc C':=\kappa _2(\mc C)$, est une composante connexe de $\partial B'_{\kappa_3(D)}$. Nous pouvons supposer que $\partial B^\flat_D \supset\sigma ^{-1}_{\mc C}(\mb T\times\{1\})$ et  $\partial B'{}^{\flat}_{\kappa_3(D)}\supset\sigma '{}^{-1}_{\mc C'}(\mb T\times\{1\})$. L'hom\'{e}o\-mor\-phisme
$r_s$ de $\mb T\times[0,1+\epsilon]$ sur $ \mb T\times[s,1+\epsilon]$, d\'{e}fini par $ r_s(p,t):=(p,s+t\frac{1+\epsilon-s}{1+\epsilon})$, $s\in[0,1]$, se rel\`{e}ve en un hom\'{e}omorphisme de $\sigma _{\mc C}^{-1}( \mb T\times[0,1+\epsilon])$ sur $\sigma _{\mc C}^{-1}( \mb T\times[s,1+\epsilon])$, qui se prolonge par l'identit\'{e} en un hom\'{e}omorphisme
$$R_s : B_D\iso B_D(s):= (B_D\setminus \sigma _{\mc C}^{-1}(\mb T\times[0,s[))\,,\quad R_{s|B^\flat_D}=\id_{B^\flat_D}\,.$$
On construit un hom\'{e}omorphismes similaire $R_s'$ de $B'_{\kappa_3(D)}$ sur $B_{\kappa_3(D)}'(s):= (B'_{\kappa_3(D)}\setminus \sigma _{\mc C'}^{-1}(\mb T\times[0,s[))$. Pour $s\in[0,1]$, posons $F_s : \mc M_\eta\to\mc M'_{\eta'}$  d\'{e}fini par :
$$
\left\{
\begin{array}{l}
         F_s(m)=\breve h_2(m)\,,\quad si \quad  m\notin B_D\,,\\
         F_s(m):= R_s'\circ \breve h_2\circ R_s^{-1}(m)\,,\quad si \quad  m\in B_D(s)\,,\\
         F_s(m):=  \sigma _{\mc C'}'{}^{-1}\circ (H_2\times id_{[0,s]})\circ \sigma _{\mc C}(m)\,,\quad si \quad  m\in \sigma _{\mc C}^{-1}(\mb T\times[0,s])\,,
       \end{array}
\right.
$$
o\`{u} $H_2(p):=\sigma _{\mc C'}'(\breve h_2(\sigma _{\mc C}^{-1}(p,0)))$. Visiblement $F_s$ est une isotopie qui v\'{e}rifie $F_0=\breve h_2$ et $F_1(\mc M_{\mc C}\cap B_D)=\mc M_{\mc C'}'\cap B'_{\kappa_3(D)}$. Pour achever la preuve du lemme, il suffit d'effectuer successivement de telles isotopies, pour toutes les composantes du bord
de chaque bloc de JSJ de $\mc M_\eta$.
\end{dem}
\begin{lema}\label{conjseif} Il existe un hom\'{e}omorphisme $h_3$ isotope \`{a} $\breve h_2$, qui satisfait les m\^{e}mes propri\'{e}t\'{e}s (\ref{accordeon}) que $\breve h_2$ et qui de plus conjugue les fibrations de Seifert du compl\'{e}mentaire des tores \'{e}paissis, i.e. il existe  des hom\'{e}omorphismes $\varsigma_D : \wh K_D\iso \wh K'_{\kappa_3(D)}$, $D\in \mf R$, tels que $\wh \rho _{\kappa_3(D)}'\circ h_{3|B_D^\flat}= \varsigma_D\circ \wh\rho _D$.
\end{lema}
\begin{dem}
Visiblement $B_{D}^\flat$ est muni de deux fibrations de Seifert : $\wh{\rho }_{D}$ de base $\wh K_D$ et $\wh\rho _{\kappa_3(D)}'\circ \breve h_{2|B_D^\flat}$ de base $\wh K_{\kappa_3(D)}'$. Comme $B_D^\flat $ n'est ni un un tore plein, ni un tore \'{e}paissi, le th\'{e}or\`{e}me d'unicit\'{e} des fibrations de Seifert  (1.2.5) de \cite{LMW}, donne une isotopie $\psi_{D,\, s} : B_D^\flat\to B_D^\flat$, $s\in [0,1]$, telle que $\psi_{D, 0}$ est l'identit\'{e} et $\psi_{D, 1}$ conjugue les feuilletages d\'{e}finies par ces deux fibrations. Pr\'{e}cis\'{e}ment, si $\varsigma_D: \wh K_D\to \wh K'_{\kappa_{3}(D)}$ est l'hom\'{e}omorphisme induit par $\psi_{D, 1}$ sur les espaces de feuilles, on a : $\wh\rho _{\kappa_3(D)}'\circ \breve h_2\circ \psi_{D, 1} = \varsigma _D\circ \wh\rho _D$. Gr\^{a}ce \`{a} l'assertion (\ref{extensa}) du lemme (\ref{extisotopies}) ci-dessous, ces isotopies se recollent en une isotopie globale
$\psi_s : \mc M_\eta\to\mc M_\eta$, v\'{e}rifiant $\psi_0=\mathrm{id}_{\mc M_\eta}$, $\psi_{s|B_D^\flat}=\psi_{D,s}$, $D\in \mf R$ et $\psi_{s|\mb T_{\mc C}}={\rm id}_{\mb T_{\mc C}}$, $\mc C\in \mf C$. Pour achever la d\'{e}monstration, il suffit de poser $h_3=\breve h_2\circ\psi_1$.
\end{dem}

\begin{lema}[d'extension des isotopies]\label{extisotopies}
Soit $B$ une vari\'{e}t\'{e} \`{a} bords et $B^\flat\subset B$ une sous-vari\'{e}t\'{e} \`{a} bords de m\^{e}me dimension, telle qu'il existe un hom\'{e}omorphisme   $\sigma$ de $\overline{B\setminus B^\flat}$ sur  $\partial B\times [0,1]$, v\'{e}rifiant  $\sigma (\partial B)= \partial B\times \{1\}$ et  $\sigma (\partial B^\flat)=\partial B\times \{0\}$. Alors :
\begin{enumerate}[(a)]
 \item\label{extensa} si $F_s : B^\flat\to B^\flat$, $s\in [0,1]$, est une isotopie telle que $F_0= \mathrm{id}_{B\flat}$, alors il existe une isotopie $F_s' : B\to B$, telle que $F'_{s\,|B^\flat}=F_s$ et $F'_{s\,|\partial B}= \mathrm{id}_{\partial B}$, $s\in [0,1]$;
  \item\label{extensb} si $G_s : \partial B\to \partial B$, $s\in [0,1]$, est une isotopie telle que $G_0=\mathrm{id}_{\partial B}$, alors il existe alors une isotopie $G'_s : B\to B$, telle que $G'_{s\,|B^\flat}=\mathrm{id}_{B^\flat}$ et $G'_{s\,|\partial B}= G_s$, $s\in [0,1]$.
\end{enumerate}
\end{lema}
\begin{dem} (\ref{extensa}) Notons $\wt F_s:=\sigma \circ F_s\circ\sigma ^{-1}_{|\partial B^{\flat}}$. Pour  $m\in \overline{B\setminus B^\flat}$, nous posons $\wt F_s'(m):=\sigma ^{-1}\circ \wt F_s'\circ\sigma (m)$,  avec $\wt F_s'(p,t):= \wt F_{s-t}(p,t)$, si $0\leq t\leq s$ et $\wt F_s'(p, t):= (p,t)$, si $s\leq t\leq 1$. La preuve de (\ref{extensb}) est similaire.
\end{dem}

Pour chaque $D\in \mf R$, l'hom\'{e}omorphisme $\varsigma_D$  donn\'{e} par (\ref{conjseif}), induit une bijection  $\varpi_D$ entre les points singuliers de $\mc D$ situ\'{e}s sur $D$ et ceux  de $\mc D'$ situ\'{e}s sur $D':=\kappa_3(D)$. Celle-ci v\'{e}rifie, avec les notations (\ref{baseseifert}) :
$$
\varpi_D(S_{\mf M}(D))=S_{\mf M}(D')\quad \hbox{\rm et donc}\quad \varpi_D(\wh{S}(D))=\wh{S}(D')\,,
$$
puisque les points d'attache des branches mortes correspondent aux fibres exceptionnelles des fibrations de Seifert  et les \'{e}l\'{e}ments de $\wh{S}(D)$, resp. de $\wh{S}(D')$, correspondent aux composantes connexes  de $\partial \wh K_D$, resp. de $\partial \wh{K}'_{D'}$. Il est facile de prouver que, quitte \`{a} modifier $h_3$ par une isotopie,
$\varsigma_D$ envoie le disque $D_s$, $s\in S_{\mf M}(D)$, sur le disque $D'_{\varpi_D(s)}$. Ainsi $h_3$ envoie le bloc \'{e}l\'{e}mentaire $\mc M_D$ sur le bloc \'{e}l\'{e}mentaire $\mc M_{D'}'$, en conjuguant les fibrations de Hopf restreintes \`{a} ces blocs.

\subsection{Conjugaison des arbres duaux des diviseurs}\label{4.3} R\'{e}capitulons les r\'{e}sultats obtenus : 
nous avons construit un hom\'{e}o\-mor\-phisme $h_3 :\mc M_\eta\to \mc M_{\eta'}'$ tel que $h_3\asymp h$, ainsi que des bijections
\begin{equation}\label{bijectioninterm}
\kappa_3 : \mf R\to\mf R\,,\quad \kappa _2 : \mf C\to \mf C'\,,\quad \kappa _1 : \mf M\to \mf M'\,,
\end{equation}
qui satisfont pour tout $\mc C\in \mf C$ et  $\wt{\mc C}\in \mf M$, les propri\'{e}t\'{e}s suivantes :
\begin{enumerate}[(a)]
\item \label{aaa} les images par $\kappa_3$ des composantes d'extr\'{e}mit\'{e} de $\mc C$, sont les composantes d'extr\'{e}mit\'{e} de $\kappa _2(\mc C)$;
\item\label{bbb} si $D_0\in \mf R$ est la composante d'attache de $\wt{\mc C}$, alors $\kappa_3(D_0)$ est la composante d'attache de $\kappa _1(\mc C)$;
\item\label{ccc} l'hom\'{e}omorphisme $h_3$ envoie $\mc M_{\mc C}$ sur $\mc M'_{\kappa _2(\mc C)}$, $\mc M_{\wt{\mc C}}$ sur $\mc M'_{\kappa _1(\wt{\mc C})}$ et $\mb T_{\mc C}$ sur $\mb T'_{\kappa _2(\mc C)}$; pour $D\in \mf R$, il envoie aussi, en conjuguant les fibrations de Seifert, $B_D$ sur $B'_{\kappa_3(D)}$, $B'{}^\flat_{\!\!D}$ sur $B'{}^{\flat}_{\!\!\kappa_3(D)}$ et $\mc M_{D}$ sur $ \mc M'_{\kappa_3(D)}$.
\end{enumerate}
La proposition suivante  \'{e}tend les correspondances (\ref{bijectioninterm}) \`{a} toutes les composantes des diviseurs. Elle pr\'{e}cise les r\'{e}sultats classiques de Zariski-Lejeune en donnant, par la propri\'{e}t\'{e} (\ref{ccc}), les relations entre l'hom\'{e}omorphisme $h$ qui transforme $S$ en $S'$ et la correspondance des arbres duaux de $\mc D$ et $\mc D'$.
\begin{prop}\label{zariskilejeune}
Il existe une bijection $\kappa : \comp(\mc D)\to \comp(\mc D')$ entre les ensembles de composantes irr\'{e}ductibles de $\mc D  $ et $\mc D'$, telle que :
\begin{enumerate}
  \item\label{un} $(\kappa(D),\kappa(D'))=(D,D')$, pour tout $D,D'\in \comp(\mc D)$;
  \item\label{deux} pour tout $\mc C\in \mf C$, resp. $\wt{\mc C}\in \mf M$, on a l'\'{e}quivalence : ($D\in \mc C)\Leftrightarrow(\kappa(D)\in \kappa _2(\mc C))$, resp. $(D\in \wt{\mc C})\Leftrightarrow(\kappa(D)\in \kappa _1(\wt{\mc C}))$; en particulier $\mc C$ et $\kappa_2(\mc C)$ ont m\^{e}me longueur, ainsi que $\wt{\mc C}$ et $\kappa _1(\wt{\mc C})$;
\item\label{trois} la restriction de $\kappa$ \`{a} $\mf R\subset\comp(\mc D)$ est \'{e}gale \`{a} $\kappa_3$.
\end{enumerate}
En particulier les propri\'{e}t\'{e}s (\ref{aaa}) (\ref{bbb}) et (\ref{ccc}) ci-dessus restent satisfaites par $\kappa$.
\end{prop}


Avant de prouver la proposition \ref{zariskilejeune} nous avons besoin d'un r\'{e}sultat auxiliaire. Soit $\mc C\in\mf C\cup\mf M$.
\begin{itemize}
\item Si $\mc C=\{D_{0},\ldots,D_{l_{\mc C}+1}\}\in\mf C$, consid\'{e}rons la cha\^{\i}ne $\mc C'=\{D_{0}',\ldots,D_{l_{\mc C'}+1}'\}=\kappa_{2}(\mc C)\in\mf C'$, que nous num\'{e}rotons  de mani\`{e}re \`{a} avoir $D_{0}'=\kappa_{3}(D_{0})$ et $D_{l_{\mc C'}+1}'=\kappa_{3}(D_{l_{\mc C}+1})$.

\item Si $\mc C=\{D_{0},\ldots,D_{l_{\mc C}}\}\in\mf M$, consid\'{e}rons $\mc C'=\{D_{0}',\ldots,D_{l_{\mc C'}}'\}=\kappa_{2}(\mc C)\in\mf M'$.
\end{itemize}
D\'{e}signons par $\mf c_j\in H_1(\mc M_{\mc C},\mb Z)$ le m\'{e}ridien associ\'{e} \`{a} $D_j$ et par $\mf c_k'\in H_1(\mc M'_{\mc C'}, \mb Z)$ celui associ\'{e} \`{a} $D'_k$, cf. (\ref{meridienchaine}). L'\'{e}galit\'{e} $h_3(\mc M_{\mc C})=\mc M'_{\mc C'}$ induit un isomorphisme :  $$h_{3\ast } : H_1(\mc M_{\mc C}, \mb Z)\to H_1(\mc M'_{\mc C'}, \mb Z)\,.$$
Dans le cas o\`{u} $\mc C$ et $\mc C'$ sont des branches mortes, $\mf c_{l_{\mc C}+1}$ et $\mf c_{l_{\mc C'}+1}'$ d\'{e}signent les m\'{e}ridiens exceptionnels correspondants.
Comme $h_{3}$ conjugue les fibrations de Seifert de $B_{D_{0}}$ et $B_{D_{0}'}'$, il conjugue aussi les fibres exceptionnelles; ainsi $h_{3}$ envoie $\mc M_{\mc C}^{\circ}$ sur $\mc M_{\mc C'}'^{\circ}$ et induit un isomorphisme :
$$h_{3*}:H_{1}(\mc M_{\mc C}^{\circ},\mb Z)\to H_{1}(\mc M_{\mc C'}'^{\circ},\mb Z)\,.$$

\begin{lema}\label{egalmeridien}
Pour tout $j=0,\ldots,l+1$, on a les \'{e}galit\'{e}s : $l_{\mc C}=l_{\mc C'}=:l$ et $h_{3*}(\mf c_{j})=\mf c_{j}'\in H_{1}^{\mc C'}$.
\end{lema}

\noindent Nous allons d'abord d\'{e}montrer le sous-lemme technique suivant\footnote{Nous remercions Mark Spivakovski   pour son aide et ses suggestions concernant la preuve de ce lemme.} :
\begin{sublema}\label{spivakovski}
Soient $a=(\alpha _1,\alpha _2)$ et  $b=(\beta _1,\beta _2)\in \mb Z^2$, $\mathrm{pgcd}(\alpha _1,\alpha _2)=1$, $\mathrm{pgcd}(\beta _1,\beta _2)=1$ tels que  $\det(a,b)>0$. Il existe alors  $n\in \mb N$ unique et une suite finie
$\mathbf{c}:=(c_0,\ldots ,c_{n+1})$ d'\'{e}l\'{e}ments de $\mb Z^2\cap(\mb Qa+\mb Qb)\subset\mb Q^2$
unique tels que :
\begin{equation}\label{concllemme}
\left\{
  \begin{array}{l}
  \det(c_j,c_{j+1})=1,\; j= 0,\ldots ,n, \\
   \det(c_{k-1}, c_{k+1})>1,\; k = 1,\ldots ,n, \\
   c_0=a,\; c_{n+1}=b\,.
  \end{array}
\right.
\end{equation}
En particulier  si $\det(a,b)=1$, l'unique suite $\mathbf{c}$ satisfaisant (\ref{concllemme}) est   donn\'{e}e par $n=0$, $c_0=a$ et $c_1=b$.
\end{sublema}
\begin{dem2}{du sous-lemme}
L'existence de $n$ et des $c_j$ se montre ais\'{e}ment \`{a} l'aide de fractions continues. Pour l'unicit\'{e}, nous utiliserons l'assertion suivante, facile \`{a} prouver\footnote{Visiblement,  $0<\lambda ,\mu <1$; ainsi $(1,1)$ appartient \`{a} l'int\'{e}rieur du parall\'{e}\-lo\-gram\-me de sommet $0$, $u$, $u+v$, $v$; mais ceci est impossible lorsque
 $u$ et $v$ forment une $\mb Z$-base de $\mb Z^2$.} :
\begin{enumerate}
\item[$(\diamond)$] \it soient $u:=(\nu_1,\nu_2), v:=(\upsilon_1,\upsilon_2)\in \mb Z^2$ et $\lambda, \mu  \in \mb Q_{>0}$, tels que $\det(u,v)> 0$, $\nu_1+\nu_2>1$, $\upsilon_1+\upsilon_2>1$ et $(1,1)= \lambda u+\mu v$. Alors $\det(u,v)>1$.
\end{enumerate}
%

Pour l'unicit\'{e}, nous raisonnons par double r\'{e}currence, avec l'hypoth\`{e}se $\mc H_{N,\,N'}$ suivante :
\begin{enumerate}
\item[] \it ``si  $\mathbf{c}:=(c_j)_{j=0}^{n+1}$ et $\mathbf{c'}:=(c'_k)_{k=0}^{n'+1}$ sont deux suites finies d'\'{e}l\'{e}\-ments de $\mb Z^2\cap(\mb Qa+\mb Qb)$, satisfaisant (\ref{concllemme}) et si $0\leq n\leq N$ et $0\leq n'\leq N'$, alors $n=n'$ et $\mathbf{c}=\mathbf{c'}$.''
\end{enumerate}
Nous allons d'abord montrer $\mc H_{0,\,N}\Rightarrow\mc H_{0,\,N+1}$; par sym\'{e}trie on aura aussi $\mc H_{N,\,0}\Rightarrow\mc H_{N+1,\,0}$; comme $\mc H_{0,\,0}$ est \'{e}vident, il suffira alors de prouver l'implication $\mc H_{N-1,\,N'-1}\Rightarrow\mc H_{N,\,N'}$.
\begin{description}
 \item [$\mc H_{0,\,N'}\Rightarrow\mc H_{0,\,N'+1}$ ]  \`A un automorphisme  de $\mb Z^2$ pr\`{e}s, nous supposons $c_0=a=(1,0)$ et $c_1=b=(0,1)$. La propri\'{e}t\'{e} $(\diamond)$ donne l'existence d'un indice $\wt k\in\{1,\ldots, N'\}$, tel que $c_{\wt k}'=(1,1)$, $1\leq i_0\leq N$. On conclut en appliquant l'hypoth\`{e}se de r\'{e}currence aux deux suites $((1,0), (1,1))$ et $(c_0',\ldots ,c_{\wt k}')$, ainsi qu'aux deux suites  $((1,1), (0,1))$ et $\mathbf{\wt c'}:= (c_{\wt k}',\ldots ,c_{N'+1}')$.
  \item [$\mc H_{N-1,N'-1}\Rightarrow\mc H_{N,N'}$ ] Toujours \`{a} un automorphisme de $\mb Z^2$ pr\`{e}s, nous pouvons maintenant supposer $a=(1,0)$ et $b=(\beta _1,\beta _2)$, avec $\beta _1<\beta _2$. Gr\^{a}ce \`{a}  $(\diamond)$ on obtient deux indices $\wt j\in\{1,\ldots,N\}$ et $\wt k\in\{1,\ldots ,N'\}$ tels que $c_{\wt j}=c_{\wt k}'=(1,1)$. L'hypoth\`{e}se de r\'{e}currence appliqu\'{e}e aux deux suites $(c_j)_{j=0,\ldots,\wt j}$ et $(c_k)_{k=0,\ldots,\wt k}$, ainsi qu'aux deux suites $(c_j)_{j=\wt j,\ldots,N}$ et $(c_k)_{k=\wt k,\ldots,N'}$, permet de conclure.
\end{description}
Ceci ach\`{e}ve la d\'{e}monstration du sous-lemme.
\end{dem2}
\begin{dem2}{du lemme (\ref{egalmeridien})}
D'apr\`{e}s le lemme (\ref{conjseif}),   $h_3$ envoie toute  composante de $\partial\mc M_{\mc C}$  sur une composante de $\partial\mc M'_{\mc C'}$ en conjuguant les fibrations de Seifert correspondantes.  L'isomorphisme $h_{3*}$ induit en homologie satisfait donc les \'{e}galit\'{e}s :
\begin{equation}\label{egalseif}
h_{3\ast}(\mf c_0)=\mf c'_0\quad \hbox{\rm et}\quad h_{3\ast}(\mf c_{l+1})=\mf c'_{l'+1}\,,
\end{equation}
o\`{u} nous posons $l:=l_{\mc C}$ et $l'=l_{\mc C'}$ pour abr\'{e}ger.
Gr\^{a}ce \`{a} l'assertion (ii) de la proposition \ref{det}  nous d\'{e}duisons aussi que
\begin{equation}
\label{h3det}
\det(\mf a,\mf b)=\det{}'( h_{3*}(\mf a),h_{3*}(\mf b))\,,
\end{equation}
pour tous $\mf a,\mf b\in H_{1}^{\mc C}$.
Des relations (\ref{formindice}) on tire : $$\det(\mf c_{j-1},\mf c_{j+1})=- (D_{j},D_{j})\geq 2\,,\quad j=1,\ldots ,l\,,$$
car l'application $E$ de r\'{e}duction de $S$ est minimale.

Posons $\mf c''_j:=h_{3\ast}^{-1}(\mf c'_j)$. Les  deux suites finies $\mathbf{c}:= (\mf c_j)_{j=0,\ldots,l+1}$ et $\mathbf{c''}:=(\mf c''_j)_{j=0,\ldots,l'+1}$ d'\'{e}l\'{e}ments de $H_1(\mc M'_{\mc C'},\mb Z)\simeq \mb Z^2$, ont m\^{e}me premier et m\^{e}me dernier terme (\ref{egalseif}); elles satisfont les relations (\ref{concllemme}) du lemme (\ref{spivakovski}). La conclusion r\'{e}sulte de l'unicit\'{e} de ces familles.
\end{dem2}

\begin{dem2}{de la proposition (\ref{zariskilejeune})}
Les cha\^{\i}nes et les branches mortes qui se correspondent par $\kappa _2$ et par $\kappa _1$, ont m\^{e}me longueur; il  existe donc une et une seule bijection $\kappa : \comp(\mc D)\to \comp(\mc D')$ qui \'{e}tend $\kappa_3$ et qui satisfait les assertions (\ref{deux}) et (\ref{trois}), ainsi que l'\'{e}quivalence $(D\cap D'\neq\emptyset\Leftrightarrow \kappa(D)\cap\kappa(D')\neq\emptyset)$. Les auto-intersections de toutes les composantes compactes consid\'{e}r\'{e}es \'{e}tant $\leq -1$, il suffit pour prouver (\ref{un}) de montrer les relations :
 \begin{equation}\label{egautoint}
 (D,D)=(\kappa(D),\kappa(D))\,,\quad \hbox{pour tout}
\quad D\in \comp(\mc E)\,.
\end{equation}
Lorsque $D_j$ est contenue dans une cha\^{\i}ne ou une branche morte de $\mc E$, les relations (\ref{formindice}) se d\'{e}duisent directement des \'{e}galit\'{e}s $(D_j,D_j)=-\det(\mf c_{j-1},\mf c_{j+1})$. Comme $h_{3\ast} $ commute aux formes d\'{e}terminants (\ref{h3det}), relations $h_{3\ast}(\mf c_j)=\mf c'_{j}$ du lemme (\ref{egalmeridien}) 
 donnent : $(D_j,D_j)=(\kappa(D_j),\kappa(D_j))$, pour tout $j=1,\ldots,l_{\mc C}+1$.

Il reste \`{a} prouver (\ref{egautoint}) lorsque $D$ est de valence $\geq 3$. Remarquons qu'alors  $\mc M_D$ est un r\'{e}tract par d\'{e}formation de $$\mc M_D^{\sharp}:=\mc M_D\cup_{j=1}^{v(D)}\mc M_{s_j}\,,\quad D\cap D_j=:\{s_j\}\,,$$
o\`{u} $D_1,\ldots ,D_{v(D)}$ sont les composantes de $\mc D$ adjacentes \`{a} $D$. Le point singulier $s_j$ est le point de branchement d'une cha\^{\i}ne, d'une branche morte  $\mc C_j$, ou bien encore d'une transform\'{e}e stricte. Consid\'{e}rons le m\'{e}ridien associ\'{e} \`{a} $D_j$, qui est un \'{e}l\'{e}ment $\mf c_j$ de $H_1(M_j,\mb Z)\simeq H_1(\mc M_{s_j}\cap \mc M_D,\mb Z)$, o\`{u} $M_j$ est le tore \'{e}paissi $\mc M_{\mc C_j}$, ou $\mc M_{\mc C_j}^\circ$ dans les deux premiers cas, ou bien encore $\mc {M}_{s_j}\cup \mc M_{D_j}$ dans le dernier cas.
D\'{e}signons par $\wt{\mf c}_j\in H_1(\mc M_D,\mb Z)$ l'image de $\mf c_j$ par le monomorphisme induit en homologie
par l'inclusion $\mc M_D\cap\mc M_{s_{j}}\subset \mc M_D$.
%
Nous pouvons r\'{e}\'{e}crire
la \emph{formule d'indice le long de} $D$ donn\'{e}e par  (\ref{h1}) de la fa\c{c}on suivante, cf. \cite{Dimca,EN} :
\begin{equation}\label{formuleintersec}
(D,D)\mf c + \sum_{j=1}^{v(D)}  \wt{\mf c}_j = 0 \quad \hbox{\rm dans}\quad H_1(\mc M_{D},\mb Z)\,,
\end{equation}
o\`{u} $\mf c$ est la classe d'homologie d'une fibre  $\rho _D^{-1}(p)$, $p\in K_D$, que nous appelons \emph{le m\'{e}ridien associ\'{e} \`{a} $D$}. D'apr\`{e}s ce qui pr\'{e}c\`{e}de,  $\{\kappa(D_j)\}_{j=1,\ldots,v(D)}$ est la collection des composantes adjacentes \`{a} $\kappa(D)$ et gr\^{a}ce au  lemme (\ref{egalmeridien}), 
leurs m\'{e}ridiens respectifs sont  $h_{3\ast}(\mf c_j)\in H_1(\mc M'_{\kappa(D)}\cap \mc M'_{s'_j},\mb Z)$, o\`{u} $\{s'_j\}:=\kappa(D)\cap\kappa(D_j)$. De m\^{e}me $h_{3\ast}(\mf c)$ est le  m\'{e}ridien de $\kappa(D)$, car $h_3$ conjugue les fibration de Seifert de $B_D$ et $B'_{\kappa(D)}$. La formule d'indice le long de $\kappa_3(D)$ donne l'\'{e}galit\'{e} (\ref{egautoint}).
\end{dem2}


\subsection{Extension \`{a} la dimension quatre}
Toujours avec les notations (\ref{KD}) et (\ref{XK}), nous d\'{e}finissons maintenant la collection des \emph{blocs \'{e}l\'{e}mentaires du 4-tube de Milnor $\mc T_\eta$} par :
\begin{equation}\label{4blocselem}
\mc T_{s}:=\mc T_{\eta}\cap\Omega _{s}\quad {\rm et} \quad\mc T_D:=\mc T_\eta(K_{D})\,,\quad s\in {\rm Sing}(\mc D)\,, \quad
D\in \comp(\mc D)\,.
\end{equation}
Un \emph{4-tube associ\'{e} \`{a} une cha\^{\i}ne $\mc C\in \mf C$}, resp. \emph{\`{a} une branche morte $\wt{\mc C}\in \mf M$} est, avec les notations (\ref{notachaine1}) et (\ref{notachaine2}), resp. (\ref{relbrmorte}) et (\ref{voisbrmorte}) :
\begin{equation}\label{4voischaine}
    \mc T_{\mc C} :=\bigcup _{j=1}^{l_{\mc C}}\mc T_{D_j}\cup\bigcup_{j=0}^{{l_{\mc C}}}\mc T_{s_j}\,,\quad{\rm resp.}\quad
\mc T_{\wt{\mc C}}:=\bigcup _{j=1}^{l_{\wt{\mc C}}}\mc T_{D_j}\cup\bigcup_{j=0}^{l_{\wt{\mc C}}-1}\mc T_{s_j}\,.
\end{equation}
Nous d\'{e}finissons de mani\`{e}re similaire les blocs \'{e}l\'{e}mentaires de $\mc T_{\eta'}'$, que nous notons $\mc T_{s'}'$, $s'\in {\rm Sing}(\mc D')$ et $\mc T'_{D'}$,
$D'\in \comp(\mc D')$, ainsi que les 4-tubes $\mc T'_{\mc C'}$, $\mc C'\in\mf C'$ et $\mc T'_{\wt{\mc C}'}$,  $\wt{\mc C}'\in \mf M'$.

Nous allons d'abord construire, pour  $\star\in \mf R$, des
hom\'{e}omorphismes $G_{\star} : \mc T_{\star}\to \mc T_{\kappa(\star)}'$  satisfaisant les propri\'{e}t\'{e}s (\ref{prohol}), (\ref{proprfibre}) de (\ref{homeoexcellents}) et co\"{\i}ncidant avec $h_3$ sur $\mc T_{\star}\cap\mc M_{\eta}=\mc M_\star$. Ensuite nous construirons  $G_{\star}$, lorsque $\star$ est une cha\^{\i}ne, puis une branche morte ou une transform\'{e} stricte, qui satisfera  toujours les propri\'{e}t\'{e}s (\ref{prohol}), (\ref{proprfibre}) de (\ref{homeoexcellents}), se recollera avec les $G_D$, $D\in \mf R$ d\'{e}j\`{a} construits, mais qui ne co\"{\i}ncidera plus n\'{e}cessairement avec $h_3$ sur $\mc M_\star$. Enfin, \`{a} l'aide de twists de Dehn,  nous modifieront l'hom\'{e}omorphisme global
\begin{equation}\label{homoprovisoire}
    G : \mc T _\eta\longrightarrow \mc T_{\eta'}'\,,\quad G_{|\mc M_{\star}} =G_{\star}\,,\quad \star\in \mf R\cup\mf C\cup\mf M\,,
\end{equation}
ainsi obtenu, pour qu'il devienne isotope \`{a} $h_3$ en restriction \`{a} $\mc M_\eta$. Nous avons maintenant un hom\'{e}omorphisme $\Phi$ qui satisfait le th\'{e}or\`{e}me \ref{teomarquagexcellent}.

\subsubsection{Construction de $G_{D}$, pour   $D \in\mf R$}
Les restrictions des fibrations de Hopf  aux blocs \'{e}l\'{e}mentaires $\mc T_D$ et $\mc T_{\kappa(D)}'$, $D\in \comp(\mc D)$,
sont des fibrations en disques globalement triviales; on dispose sur ces blocs de champs de vecteurs diff\'{e}rentiables $Z$ et $Z'$, tangents aux fibres de Hopf qui, en restriction \`{a} chaque fibre, correspondent au champ radial r\'{e}el $u\frac{\partial}{\partial u}+ v\frac{\partial}{\partial v}$, dans des coordonn\'{e}es trivialisantes $(u+iv,\rho _D) : \mc T_D\to  \mb D_1\times K_D$. Nous d\'{e}finissons un hom\'{e}omorphisme   qui \'{e}tend  $h_{3|\mc M_D}$ et conjugue les fibrations de Hopf, en posant :
$$
\left\{
  \begin{array}{l}
G_D : \mc T_D\to \mc T_{\kappa(D)}'\,,\quad G_{D|\mc M_D}=h_3\,,\quad \rho '_{\kappa(D)}\circ G_D= \rho _{D|\mc T_D}\,,\\
   G_D(\phi_t^Z(m)):=\phi_{t}^{Z'}(h_3(m))\,,\quad\hbox{si}\quad m \in \mc M_D, t<0\,, \\
 G_D(m):=\varsigma_D(m)(h_3(m))\,,\quad\hbox{si}\quad m\in K_D\,,
  \end{array}
\right.
$$
o\`{u} $\phi_t^Z$ et $\phi_t^{Z'}$ d\'{e}signent les flots de $Z$ et  $Z'$ respectivement.

\subsubsection{Construction de $G_{\mc C}$, lorsque $\mc C$ est une
cha\^{\i}ne} Consi\-d\'{e}\-rons  une cha\^{\i}ne $\mc C\in \mf C$ de $\mc D$ et la cha\^{\i}ne de $\mc D'$ associ\'{e}e, $\mc C':=\kappa _2(\mc C))$,
$$\mc C = \{D_j\}_{j=0,\ldots, l+1}\in \mf C\,,\quad
\mc C' = \{D'_j\}_{j=0,\ldots, l+1}\,,\quad D'_j:=\kappa(D_j)\,,$$
les composantes $D_0$, $D_{l+1}$ \'{e}tant de valences $\geq 3$. Nous supposerons $l\geq 1$, le cas d'une cha\^{\i}ne de longueur $l=0$, n'ayant aucune composante de valence $2$ et avec un seul
point singulier $\{s\}=D_0\cap D_1$ se traitant de mani\`{e}re identique, avec  $\mc M_{\mc C}=\mc M_{s}$ et $\mc M_{\mc C'}'=\mc M_{s'}'$, $\{s'\}:=D'_0\cap D'_1$.

Dans une premi\`{e}re \'{e}tape, nous allons construire  des hom\'{e}omorphismes $g_{s_j}$ holomorphes sur des voisinages $W_{s_j}$ des singularit\'{e}s $\{s_j\}:=D_{j-1}\cap D_j$; ensuite nous construirons des hom\'{e}omorphismes $g_{D_j}$ sur les blocs \'{e}l\'{e}mentaires  $\mc T_{D_j}$, qui conjuguent les fibrations de Hopf; enfin, \`{a} la derni\`{e}re \'{e}tape, nous recollons ces hom\'{e}omorphismes, pour obtenir un hom\'{e}omorphisme
\begin{equation}\label{homeochaine}
G_{\mc C} : \mc T_{\mc C}\to \mc T'_{\mc C'}
\end{equation}
qui satisfait les propri\'{e}t\'{e}s (\ref{prohol}) et (\ref{proprfibre}) des hom\'{e}omorphismes excellents (\ref{homeoexcellents}).\\

\indent \textit{\'Etape 1. }
L'application $f\circ E$, compos\'{e}e de l'\'{e}quation de $S$ fix\'{e}e \`{a} la section (\ref{ssectubesMilnor}) avec l'application de r\'{e}duction, est une \'{e}quation globale de $\mc D$. Le corollaire \ref{homologie} donne des formules universelles (voir aussi \cite[Theorem 18.2]{EN}) exprimant les multiplicit\'{e}s $\nu_D(f\circ E)$ le long de chaque composante $D$ de $\mc D$, \`{a} partir de la matrice d'intersection de $\mc D$.
%
%
Les matrices d'intersection $(D',D'')$ et $(\kappa(D'),\kappa(D''))$, $D'$, $D''\in \comp(\mc D)$,  sont \'{e}gales  d'apr\`{e}s  l'assertion (\ref{un}) de (\ref{zariskilejeune}).
On a donc en particulier, toujours avec les notations de la section (\ref{ssectubesMilnor}),  $$\nu_{D_j'}(f'\circ E')= \nu_{D_j}(f\circ E)=: m_j\,,\quad j=0,\ldots ,l+1\,.$$ Soit $s_j$ le point d'intersection de  $D_j$ et $D_{j+1}$ et $s_j'$ celui de $D_j'$ et $D_{j+1}'$. Il existe en ces points des coordonn\'{e}es holomorphes locales
\begin{equation}\label{coordhomo}
(u_j, v_j) : W_{s_j}\iso \mb D_1\times \mb D_1\,,\quad (u_j',v_j') :  W'_{s'_j}\iso \mb D_1\times \mb D_1\,,
\end{equation}
avec $W_{s_j}\subset \inte {\Omega} _{s_j}$ et $W'_{s'_j}\subset \inte {\Omega}{}' _{s'_j}$,
telles que $v_j=0$, resp. $v'_j=0$, soit une \'{e}qua\-tion locale de $D_j$, resp. de $D'_j$ et qui rendent $f\circ E$ et $f'\circ E'$ monomiaux :
$$f\circ E_{|W_{s_j}}=u^{m_{j+1}}_j v^{m_j}_j \quad {\rm et}\quad f'\circ E'_{|W'_{s'_j}}=\frac{\eta}{\eta'}\,u'{}^{m_{j+1}}_j v'{}^{m_{j}}_j,.$$ On obtient ainsi un diff\'{e}o\-mor\-phisme holomorphe $g_{s_j}$ entre les vari\'{e}t\'{e}s \`{a} bords et \`{a} coins $W_{s_j}\cap \mc T_{s_j}=W_{s_j}\cap \mc T_\eta$ et   $W'_{s'_j}\cap \mc T'_{s_j'}=W'_{s'_j}\cap \mc T'_{\eta'}$, en posant  :
\begin{equation}\label{bihololocal}
    g_{s_j}:=(u'_j,v'_j)^{-1}\circ(u_j,v_j) : W_{s_j}\cap \mc T_\eta \longrightarrow W'_{s'_j}\cap \mc T'_{\eta '}\,.
\end{equation}
En supposant $\eta>0$  suffisamment petit, la 3-vari\'{e}t\'{e} $W_{s_j}\cap \mc M_\eta$, ainsi que les composantes connexes $\mf T_j$ et $\mf T_{j+1}$ de $\overline{\mc M_{s_j}\setminus W_{s_j}}$,
 avec $\mf T_j\cap\mc M_{D_j}\neq\emptyset$, sont des tores \'{e}paissis. Leurs inclusions dans $\mc M_{\mc C}$ sont des isomorphismes en homologie. Supposons  que $W'_{s'_j}\cap \mc M'_{\eta '}$ satisfait les m\^{e}mes propri\'{e}t\'{e}s. Quitte \`{a} diminuer  $\eta'>0$, la restriction de $g_{s_j}$ \`{a} $W_{s_j}\cap \mc M_\eta$, \`{a} valeurs dans $W'_{s'_j}\cap \mc M'_{\eta '}$, d\'{e}finit alors un isomorphisme, not\'{e} :
 \begin{equation}\label{isog}
    g_{s_j\,\ast} : H_1(\mc M_{\mc C},\mb Z)\to H_1(\mc M'_{\mc C'},\mb Z)\,.
 \end{equation}
\begin{lema}\label{egmeridiensholo} Soient $\mf c_j$ et $\mf c'_j$, les m\'{e}ridiens associ\'{e}s aux composantes $D_j$ et $D'_j$ respectivement, cf. (\ref{meridienchaine}). Alors
$g_{s_j\,\ast}(\mf c_k)=\mf c'_{k}$, $k:=j,j+1$, pour tout $j=0,\ldots , l$.
\end{lema}
\begin{dem} Supposons $k=j$, le cas $k=j+1$ se traite de la m\^{e}me mani\`{e}re. Quitte \`{a} permuter les coordonn\'{e}es du syst\`{e}me local, nous supposons aussi  que $y_{s_j}=0$ est  une \'{e}quation de $D_j$. Pour $\eta>0$ assez petit, les fibres de $x_{s_j}$ et de $u_j$ sont transverses en tout point de $W_{s_j}\setminus D_{j+1}$, aux \emph{fibres de Milnor}, i.e. aux fibres de $f\circ E$. Sur $D_j\cap W_{s_j}$, le champ de vecteurs holomorphe  qui s'\'{e}crit $u_j\frac{\partial}{\partial u_j}$, se rel\`{e}ve donc (via l'application $u_j$) en un champ $Z$ tangent aux fibres de Milnor -et donc aussi \`{a} $\mc M_{\eta}$. On construit facilement une fonction \`{a} support compact $\alpha  :\inte{W}_{s_j}\cap \mc M_\eta \to \mb R$, telle que le flot au temps 1 de $\alpha Z$ envoie $\mf l_j:=  u_j^{-1}(p)\cap W_{s_j}\cap \mc M_\eta$ sur  $\mf l'_j:= x_{s_j}^{-1}(p)\cap W_{s_j}\cap \mc M_\eta$, o\`{u} $p\neq s_j$ d\'{e}si\-gne un point fix\'{e} de $D_j\cap \inte W_{s_j}$. Soit $K\subset D_j\cap \inte W_{s_j}$, $s_j\in \inte K$,  $p\notin K$, un disque conforme ferm\'{e}. Quitte \`{a} diminuer $\eta>0$, la restriction de $x_{s_j}$ \`{a} $\mc M_{s_j}\setminus x_{s_j}^{-1}(K)$ est encore une fibration en cercles triviale; ainsi une fibre $\mf l''_j$ de la restriction de $\rho_{D_j}$ \`{a} $\mc M_{D_j}\cap \mc M _{s_j}\subset\partial(\mc M_{s_j}\setminus x_{s_j}^{-1}(K))$, est homologue \`{a}  $\mf l'_j$. Finalement on obtient :
$$
 [\mf l_j]= [\mf l'_j]=[\mf l''_j] = \mf c_j\in H_1(\mc M_{\mc C},\mb Z)\,.
$$
Supposons  $\eta'>0$ assez petit, pour que les \'{e}galit\'{e}s similaires dans  $\mc M'_{\mc C'}$ soient satisfaites. Pour achever la d\'{e}monstration, il suffit de remarquer que par construction, $g_{s_j}$ envoie les fibres de $u_j$, resp. de $v_j$, sur les fibres de $u'_j$, resp. de $v'_j$.
\end{dem}

\indent \textit{\'Etape 2. }
Donnons-nous maintenant des hom\'{e}omorphismes
\begin{equation}\label{homeodivis}
g_{D_j} : \mc T_{D_j} \longrightarrow \mc T'_{D'_j}\,,\quad j=1,\ldots , l
\end{equation}
tels que :
\begin{enumerate}[(a)]
\item\label{egbord} $g_{D_j}(\mc T_{D_{j}}\cap \mc T_{s_j})=\mc T_{D_{j}'}'\cap\mc T'_{s'_j}$,
  \item\label{eghopf} $g_{D_j}$ conjugue les fibrations de Hopf : il existe un hom\'{e}omorphisme $\varsigma_{D_j} :K_{D_j}\to K'_{D'_j}$, tel que  $\varsigma_{D_j}\circ\rho _{D_j}(m)= \rho' _{D'_j}\circ g_{D_j}(m)$, $m\in \mc T_{D_j}$,
      \item\label{conjj} le morphisme $g_{D_j\,\ast} : H_1(\mc M_{\mc C},\mb Z)\to H_1(\mc M'_{\mc C'},\mb Z)$ induit\footnote{via les identifications $H_1(\mc M_{D_j},\mb Z)\simeq H_1(\mc M_{\mc C},\mb Z)$ et $H_1(\mc M'_{D'_j},\mb Z)\simeq H_1(\mc M'_{\mc C'},\mb Z)$ donn\'{e}s par les inclusions.} par la restriction de $g_{D_j}$ \`{a} $\mc M_{D_j}$, \`{a} valeurs dans $\mc M'_{D'_j}$, v\'{e}rifie : $g_{D_j\,\ast} (\mf c_k)=\mf c'_k$, $k=j\pm 1$, o\`{u} $g_{D_j\,\ast}$.
\end{enumerate}
Remarquons que l'\'{e}galit\'{e} (\ref{conjj}) pour $k=j$, se d\'{e}duit de (\ref{eghopf}) et que le cas $k=j-1$ est \'{e}quivalent au cas $k=j+1$,  d'apr\`{e}s les formule d'indices (\ref{formindice}) et l'assertion (\ref{un}) de  (\ref{zariskilejeune}). Ainsi, la construction de $g_{D_j}$ se fait sans peine, apr\`{e}s avoir trivialis\'{e} les fibrations de Hopf.\\

\indent \textit{\'Etape 3. } Il reste \`{a} construire, pour chaque composante connexe $\mf T$ de
$\mc T_{s_j}\setminus W_{s_j}$,  $j=0,\ldots,l$, un hom\'{e}o\-mor\-phisme d\'{e}fini sur $\mf T$,  \`{a} valeur sur une composante connexe $\mf T'$ de $\mc T'_{s'_j}\setminus W'_{s'_j}$, qui se recolle avec  $g_{s_j}$ et $g_{D_{j'}}$, $j' =j$ ou $j+1$. Pour cela fixons  un hom\'{e}omorphisme $\Lambda $ de $\mf T$ sur $[0,1]\times \mb S^1\times\mb D_1$. Donnons-nous aussi une fibration en disques ${\rho _{\mf T}} : \mf T \to C_{\mf T} := D_{j'}\cap\mf T$, qui co\"{\i}ncide avec $\rho _{D_{j'}}$ sur une composante connexe de $\Lambda ^{-1}(\{0,1\}\times \mb S^1\times\mb D_1)$ et qui, sur l'autre composante, co\"{\i}ncide avec une coordonn\'{e}e homog\'{e}n\'{e}isante  d\'{e}finie en (\ref{coordhomo}). On proc\`{e}de de m\^{e}me avec $ \mf  T'$. On constate que les restrictions de $g_{s_j}$ et $g_{D_{j'}}$ \`{a} $\partial \mf T$, conjuguent les fibrations construites. Pour conclure, il suffit d'appliquer le lemme suivant, en utilisant pour cela le lemme (\ref{egmeridiensholo}).

\begin{lema}\label{extenauxtore}
Soient $\phi_0$ et $\phi_1$ deux hom\'{e}omorphismes du tore plein  $\mb S^1\times \mb D_1$ sur lui m\^{e}me, qui commutent \`{a} la premi\`{e}re projection, i.e. $\phi _k( \theta, z)= (\theta, \underline{\phi}_k (\theta,z))$, $k=0,1$. Si leurs restrictions \`{a} $\mb S^1\times \partial \mb D_1$ induisent l'identit\'{e} en homologie, il existe  un hom\'{e}omorphisme  $\Phi$ de $[0,1]\times\mb S\times \mb D_1$ sur lui-m\^{e}me, qui commute aux deux premi\`{e}res projections, i.e. $\Phi (\theta, z, t)= (t,\theta, \underline{\Phi}_t(\theta, z))$ et tel que : $\Phi_0 =\underline{\phi}_0$, $\Phi_1 =\underline{\phi}_1$, et $\underline{\Phi} _t(\theta, z)=(\theta, z)$ si $\frac{1}{3}\leq t\leq\frac{2}{3}$.
\end{lema}
\begin{dem} Les applications continues $\wt\phi _k : \theta \mapsto \underline{\phi }_k(\theta,\cdot)$, $k=0,1$, de $\mb S^1$ dans le groupe  $\mr{Aut}(\mb S^1)$  des hom\'{e}omorphismes (pr\'{e}servant l'orientation) de  $\mb S^1$ dans lui-m\^{e}me, sont homotopes. En effet  l'application $$(\psi_{\tau})_{\tau\in [0,1]}\mapsto \frac{1}{2i\pi }\int_{\tau\mapsto \psi_\tau(1)}\frac{dz}{z}$$ est un isomorphisme du groupe fondamental de $\mr{Aut}(\mb S^1)$ sur $\mb Z$; or $$\int_{\tau\mapsto \phi _k(e^{2i\pi\tau}, \,1)}\frac{dz}{z}=0,$$ car l'automorphisme de $H_1(\mb S^1\times \mb S^1,\mb Z)$ induit par $\phi _k$ est l'identit\'{e}.
Fixons  des homotopies $t\mapsto \wt\Phi _{k,\,t}\in C^0(\mb S^1, \mr{Aut}(\mb S^1))$, $t\in[0,1]$, $\wt \Phi _{k,\,0}=\wt\phi _k$, $\wt\Phi _{k,\,1}=(\theta\mapsto id_{\mb S^1})$. Il suffit de poser :
$$\underline{\Phi}_t (\theta,z)  = \left\{
\begin{array}{l}
\wt\Phi_{0,\,|z|+ 3t-1}(\theta)( \frac{z}{|z|} )\,,\quad\hbox{\rm si}\quad 0\leq 1-3t\leq |z|\leq 1 \,,\\
\underline{\phi}_0(\theta, \frac{z}{1-3t})\,,\quad \hbox{\rm si}\quad 0\leq |z|\leq 1-3t\,,\\
(\theta, z)\,,\quad \hbox{\rm si }\quad \frac{1}{3}\leq t\leq \frac{2}{3} \,,\\
\underline{\phi}_1(\theta, \frac{z}{3t-2})\,,\quad \hbox{\rm si}\quad 0\leq |z|\leq 3t-2\,,\\
\wt\Phi_{1,\,|z|+ 2-3t}(\theta)( \frac{z}{|z|} )\,,\quad\hbox{\rm si}\quad 0\leq 3t-2\leq |z|\leq 1\,.
\end{array}
\right.$$
\end{dem}

\subsubsection{Construction de $G_{\mc C}$, lorsque $\mc C$ est une branche morte ou une transform\'{e}e stricte } Consid\'{e}rons d'abord le cas d'une  une branche morte de $\mc D$, not\'{e}e $\mc C = \{D_j\}_{j=0,\ldots ,l}$, $v(D_0)\geq 3$ et d\'{e}signons par $\mc C':=\kappa _1(\mc C)=\{D'_j\}_{j=0,\ldots ,l}$, $D'_j:=\kappa(D_j)$, la branche morte de $\mc D'$ correspondante. Nous pouvons encore effectuer, dans ce contexte, toutes les constructions pr\'{e}c\'{e}dentes, sauf pour la composante d'extr\'{e}mit\'{e} : pour   $\{s_j\}:=D_j\cap D_{j+1}$, $j=0,\ldots ,l-1$, nous construisons, avec les m\^{e}mes notations qu'en  (\ref{bihololocal}), un hom\'{e}omor\-phis\-me  $g_{s_j}$  et, pour chaque composante de valence deux, un hom\'{e}omorphisme $g_{D_j}$ comme en (\ref{homeodivis}). Dans $H_1(\mc M'{}^\circ_{\mc C'}, \mb Z)$, nous avons encore les \'{e}galit\'{e}s  $$g_{s_j\,\ast}(\mf c_k)=\mf c'_k,\quad k=j, j+1,\quad j=0,\ldots, l-1,$$
 pour les m\^{e}mes raisons qu'au  lemme (\ref{egmeridiensholo}) et gr\^{a}ce \`{a} (\ref{egalmeridien});  les  $g_{s_j}$ et $g_{D_j}$ se recollent donc,  comme \`{a} l'\'{e}tape 3 ci-dessus. Il ne reste plus qu'\`{a} \'{e}tendre $g_{s_{l-1}}$ le long de $D_l$.
Pour cela, nous supposerons comme pr\'{e}c\'{e}demment  que $\eta,\eta'>0$ sont assez petits pour que les composantes connexes de $\overline{\mc T_{s_{l-1}}\setminus W_{s_{l-1}}}$  et $\overline{\mc T_{s_{l-1}'}'\setminus W_{s_{l-1}'}'}$
soient des tores \'{e}paissis.
Il suffit alors de
construire un hom\'{e}omorphisme $g$ de la composante connexe $\mf T$
de
$(\mc T_{s_{l-1}}\setminus W_{s_{l-1}})\cup \mc T_{D_l}$ contenant $D_{l}$, sur la composante connexe $\mf T '$ de $(\mc T'_{s_{l-1}}\setminus W'_{s_{l-1}})\cup \mc T'_{D'_l}$ contenant $D_{l}'$, qui co\"{\i}ncide avec $g_{s_{l-1}}$ sur le tore plein $\mf T\cap W_{s_{l-1}}$. Fixons encore des fibrations $\rho _{\mf T} : \mf T\to\mf T \cap D_l$ et $\rho '_{\mf T'}: \mf T'\to \mf T'\cap D'_l$, qui co\"{\i}ncident avec les fibrations de Hopf sut $\mc T_{D_l}$, resp. sur $\mc T'_{D_l}$ et avec une coordonn\'{e}e homog\'{e}n\'{e}isante (\ref{coordhomo})   sur $\mf T\cap W_{s_{l-1}}$, resp. sur $\mf T'\cap W'_{s_{l-1}}$. Visiblement $\mf T$ et  $\mf T'$ sont hom\'{e}omorphes \`{a} $\mb D_1\times \mb D_1$, les fibrations $\rho _{\mf T}$ et $\rho '_{\mf T}$ correspondant \`{a} la premi\`{e}re projection. Pour achever la construction de $G_{\mc C}$, il suffit d'utiliser le lemme suivant dont la d\'{e}monstration est similaire \`{a} celle de  (\ref{extenauxtore}).
\begin{lema}
Soit $\phi $ un hom\'{e}omorphisme de $\partial \mb D_1\times \mb D_1$ sur lui m\^{e}me, qui commute \`{a} la premi\`{e}re projection : $\phi (\theta, p)=(\theta,\underline{\phi }(\theta,p))$ et qui, en restriction \`{a} $\partial \mb D_1\times \partial \mb D_1$, induit l'identit\'{e} en d'homologie. Alors $\phi $ se prolonge en un hom\'{e}omorphisme $\Phi $ de $\mb D_1\times \mb D_1$ sur lui-m\^{e}me, qui commute aussi \`{a} la premi\`{e}re projection.
\end{lema}
\begin{dem} Comme pour (\ref{extenauxtore}), il existe une application continue $t\mapsto \wt\Phi _t\in C^0(\mb S^1, \mr{Aut}(\mb S^1))$, $t\in[0,1]$, telle que $\wt \Phi _0(\theta)(\vartheta)=\vartheta $ et $\wt \Phi _1(\theta)(\vartheta)=\phi(\theta,\vartheta) $.
On pose $\Phi (z',z''):=(z', \underline{\Phi} (z',z''))$, avec :
$$\underline{\Phi} (z',z''):=\left\{
\begin{array}{l}
|z''|\cdot \wt\Phi _{|z'|}(\frac{z'}{|z'|})(\frac{z''}{|z''|})\,,\quad {\rm si}\quad |z'|\leq |z''|\leq 1\,,\\
|z''|\cdot \wt\Phi _{1+|z'|-\frac{|z''|}{|z'|}}(\frac{z'}{|z'|})(\frac{z''}{|z''|})
\,,\quad {\rm si}\quad |z'|^2\leq |z''|\leq |z'|\,,\\
|z'|^2\cdot \phi (\frac{z'}{|z'|}, \frac{z''}{|z'|^2})\,,\quad {\rm si}\quad |z''|\leq |z'|^2\leq 1\,.
\end{array}
\right.$$

\end{dem}

Consid\'{e}rons maintenant $D_1$ et $D'_1 := \kappa(D_1)$, des transform\'{e}es strictes  de composantes irr\'{e}ductibles de $S$ et $S'$ respectivement. Les composantes adjacentes $D_0\in \comp(\mc D)$, resp. $D'_0:=\kappa(D_0)\in \comp(\mc D')$, sont de valence $\geq 3$. Notons $\{s\}:=D_0\cap D_1$ et $\{s'\}:= D'_0\cap D'_1$, $\mc C:=\{D_0,D_1\}$,  $\mc C':=\{D'_0,D'_1\}$ et posons : $\mc M_{\mc C}:=\mc M_{s}\cup\mc M_{D_1}$, $\mc T_{\mc C}:=\mc T_s\cup\mc T_{D_1}$, $\mc M_{\mc C'}':=\mc M_{s'}'\cup\mc M_{D'_1}'$ et $\mc T'_{\mc C'}:=\mc T'_{s'}\cup\mc T'_{D'_1}$. Avec les m\^{e}mes notations, nous construisons comme en (\ref{isog}) un biholomorphisme $g_s : W_s\cap\mc T_{\eta}\to W'_{s'}\cap\mc T_{\eta'}'$. Pour les m\^{e}mes raisons qu'en (\ref{egmeridiensholo}), celui-ci v\'{e}rifie le \'{e}galit\'{e}s $g_{s\,\ast}(\mf c_k)=\mf c'_k$, $k=0,1$, o\`{u} $\mf c_k$, resp. $\mf c'_k$, sont les classes dans $H_1(\mc M_{\mc C},\mb Z)$, resp. dans $H_1(\mc M'_{\mc C'},\mb Z)$,  d'une fibre quelconque  de la fibration de Hopf $\rho _{D_k}$ restreinte \`{a}  $\mc M_s\cap \mc M_{D_k}$,  resp.  $\rho ' _{D'_k}$ restreinte \`{a}  $\mc M'_{s'}\cap \mc M'_{D'_k}$. Remarquons  que la restriction de $h_3$ \`{a} $\mc M_{\mc C}\cap\partial \mc B$ (qui est une composante du bord de $\mc M_\eta$), \`{a} valeurs dans $\mc M'_{\mc C'}\cap \partial\mc B'$, v\'{e}rifie aussi l'\'{e}galit\'{e}\footnote{avec les identifications donn\'{e}es par les inclusions :  $H_1(\mc M_{\mc C}\cap\partial \mc B,\mb Z)\simeq H_1(\mc M_{\mc C},\mb Z)$ et  $H_1(\mc M'_{\mc C'}\cap\partial \mc B',\mb Z)\simeq H_1(\mc M'_{\mc C'},\mb Z)$.}
$$h_{3\,\ast}(\mf c_k)=\mf c'_k\quad {\rm dans}\quad H_1(\mc M'_{\mc C'},\mb Z)\,,\quad  \,,\quad k=0,1\,.$$
En effet, par construction $h_{3}$ et $h$ sont fondamentalement \'{e}quivalents; leurs actions sur $\Gamma$ diff\`{e}rent donc d'un automorphisme int\'{e}rieur. En passant \`{a} l'homologie $h_{*}=h_{3*}$.
Le th\'{e}or\`{e}me (\ref{pk}) affirme que l'image par
$h_{*}$ du m\'{e}ridien $\mf m_{D_{1}}$ du sous-groupe p\'{e}riph\'{e}rique $\mc P\subset\Gamma$ associ\'{e} \`{a} $\mc C$  n'est autre que le m\'{e}ridien
$\mf m_{D_{1}'}\in\mc P'\subset\Gamma'$. Comme les isomorphismes $\mc P\cong H_{1}(\mc M_{\mc C},\mb Z)$ et $\mc P'\cong H_{1}(\mc M_{\mc C'}',\mb Z)$ font correspondre $\mf m_{D_{1}}$ \`{a} $\mf c_{1}$ et $\mf m_{D_{1}'}'$ \`{a} $\mf c_{1}'$ on obtient l'\'{e}galit\'{e} $h_{3*}(\mf c_{1})=\mf c_{1}'$.
D'autre part, d'apr\`{e}s la remarque \ref{incompseifert}, l'inclusion naturelle $H_{1}(\mc M_{\mc C},\mb Z)\hookrightarrow H_{1}(B_{D_{0}})$ envoie $\mf c_{0}$ sur la classe d'homologie de $\mf c_{D_{0}}\in\pi_{1}(B_{D_{0}})\subset\Gamma$ repr\'{e}sent\'{e}e\footnote{Ici nous utilisons  que les d\'{e}singularisations de $S$ et $S'$ sont minimales et donc $v(D_{0})=v(D_{0}')\ge 3$.} par une fibre de la fibration de Seifert de $D_{0}$. Nous avons une description analogue pour $\mc M_{\mc C}'$.
Comme $h_{3}$ conjugue les fibrations de Seifert de $B_{D_{0}}$ et $B_{D_{0}'}'$ il en r\'{e}sulte que $h_{3*}(\mf c_{0})=\mf c_{0}'$.

 Soit $H_{D_1} : \mc T_{D_1}\to \mc T'_{D'_1}$ un hom\'{e}omorphisme qui co\"{\i}ncide avec $h_3$ en restriction \`{a} $\mc M_{\mc C}\cap \mc B$ et qui commute aux fibrations de Hopf : $H_{D_1}(K_{D_1})=K'_{D'_1}$ et  $H_{D_1}\circ \rho _{D_1} =\rho '_{D'_1}\circ H_{D_1}$. Comme \`{a} l'\'{e}tape 3 pr\'{e}c\'{e}dente, nous construisons un hom\'{e}omorphisme $G_{\mc C} : \mc T_{\mc C}\to \mc T'_{\mc C'}$ qui \'{e}tend $g_s$, qui est \'{e}gal \`{a} $H _{D_0}$ en restriction \`{a} $\mc M_{\mc C}\cap \mc M_{D_0}$ et \`{a} $H_{D_1}$ en restriction \`{a} $\mc M_{\mc C}\cap\partial \mc B$.

\subsubsection{Modification par twists de Dehn.}\label{dehn} Nous allons maintenant
modifier l'hom\'{e}omorphisme $G$ obtenu par recollement (\ref{homoprovisoire}), en le composant  \`{a} droite par un hom\'{e}omorphisme $\Psi :\mc T_\eta\to \mc T_\eta$ qui vaut l'identit\'{e} sur chaque bloc  $\mc T_{D}$, $D\in \mf R$ et tel que $G\circ \Psi$ satisfait le th\'{e}or\`{e}me (\ref{teomarquagexcellent}). En posant $\Psi_{\mc C}:=\Psi_{|\mc T_{\mc C}}$,  il suffit de prouver l'assertion suivante, pour  tout $\mc C =:\{D_{j}\}^l_{j=0}$ d\'{e}signant une cha\^{\i}ne de $\mf C$, une branche morte, ou bien une paire de composantes associ\'{e}e \`{a} une transform\'{e}e stricte.

\begin{enumerate}
  \item[$(\star\star)$] \it Il existe un hom\'{e}omorphisme $\Psi_{\mc C}: \mc T_{\mc C}\to \mc T_{\mc C}$, $\Psi_{\mc C}(\mc T_{\mc C}\cap\mc D)=\mc T_{\mc C}\cap\mc D$, \`{a} support dans l'int\'{e}rieur de $({\Omega}_{s_0}\setminus \{s_0\})$, $\{s_0\}:=D_0\cap D _1$, tel que   $\Psi_{\mc C|\mc M_{\mc C}}$ et
   $G^{-1}\circ h_3 : \mc M_{\mc C}\to \mc M_{\mc C}$  sont homotopes relativement au bord de  $\mc M_{\mc C}$, i.e. il existe une homotopie $F_t : \mc M_{\mc C}\to \mc M_{\mc C}$, $t\in [0,1]$, telle que $F_0= G^{-1}\circ h_3$, $F_1=\Psi_{|\mc M_{\mc C}}$ et $F_t(m)=m$, pour tout  $t\in [0,1]$ et  $m\in \partial \mc M_{\mc C} $.
\end{enumerate}
Rappelons qu'\`{a} toute application continue $K$ d'une vari\'{e}t\'{e} \`{a} bord $X$ dans elle m\^{e}me, qui vaut l'identit\'{e} en restriction \`{a} un sous-ensemble $A\subset  X$, est associ\'{e} un \emph{morphisme de variation relative \`{a} $A$}, cf. \cite{AVG2} :
$$
{\rm var}_K : H_1(X, A ; \mb Z)\to H_1(X, \mb  Z)\,, \quad [\delta] \mapsto [K(\delta)-\delta]\,.
$$
Celui-ci est un invariant de la classe d'homotopie relative \`{a} $A$, de $K$. Notons que si $K_{\ast } : H_1(X,\mb Z)\to H_1(X,\mb Z)$ d\'{e}signe le morphisme induit par $K$ et $i_\ast : H_1(X,\mb Z)\to H_1(X,A; \mb Z)$ celui induit par l'inclusion $(X,\emptyset) \subset (X,A)$, on a l'identit\'{e} : $K_\ast  = {\rm id}_{H_1(X,\mb Z)} + {\rm var}_K\circ i_\ast$. Nous utiliserons le
r\'{e}sultat suivant.
\begin{prop}\label{eleinparticulier}
Deux hom\'{e}omorphismes $\chi_0$ et $\chi_1 : \mc M_{\mc C}\to \mc M_{\mc C}$ \'{e}gaux \`{a} l'identit\'{e} en restriction \`{a} $\partial \mc M_{\mc C}$, sont homotopes relativement \`{a} $\partial \mc M_{\mc C}$, si et seulement si leurs morphismes de variation sont \'{e}gaux : $${\rm var}_{\chi_0}={\rm var}_{\chi_1} : H_1(\mc M_{\mc C}, \partial \mc M_{\mc C} ; \mb Z)\to H_1(\mc M_{\mc C}, \mb Z)\,.$$
\end{prop}
Remarquons que si $\mc C$ est une branche morte, alors $(\mc M_{\mc C},\partial \mc M_{\mc C})$ est hom\'{e}o\-mor\-phe \`{a} $(\mb S^1\times \mb D_1, \mb S^1\times \mb S^1)$ et $H_1(\mc M_{\mc C}, \partial \mc M_{\mc C} ; \mb Z)=0$. Pour obtenir $(\star\star)$, on pose alors $\Psi_{\mc C}={\rm id}_{\mc T_{\mc C}}$.

Si $\mc C$ n'est pas une branche morte, l'assertion $(\star\star)$ d\'{e}coule imm\'{e}\-dia\-te\-ment du lemme suivant.
\begin{lema}\label{realvar} Supposons que  $\mc C$ est une cha\^{\i}ne ou est associ\'{e} \`{a} une transform\'{e}e stricte. Alors pour tout morphisme de $ L : H_1(\mc M_{\mc C}, \partial \mc M_{\mc C} ; \mb Z)\to H_1(\mc M_{\mc C}, \mb Z)$, il existe un  hom\'{e}omorphisme $\Psi : \mc T_{\mc C}\to \mc T_{\mc C}$ \`{a} support dans  $\inte{\Omega}_{s_0}\setminus \{s_0\}$, v\'{e}rifiant $\Psi_{\mc C}(\mc T_{\mc C}\cap\mc D)=\mc T_{\mc C}\cap\mc D$ et tel que $L$ soit le morphisme de variation de la restriction de $\Psi$ \`{a} $\mc M_{\mc C}$ :   $L={\rm var}_{\Psi_{|\mc M_{\mc C}}}$.
\end{lema}
\begin{dem2}{du lemme (\ref{realvar})} Visiblement
$H_1(\mc M_{\mc C},\partial \mc M_{\mc C},\mb Z)= \mb Z\mf d
$  est engendr\'{e} par la classe d'un chemin quelconque $\delta$  reliant les deux composantes connexes de $\partial \mc M_{\mc C}$. Gr\^{a}ce \`{a}  la formule\footnote{En effet, ${\rm var}_{\chi_{1}\chi_{2}}\mf d=[\chi_{1}\chi_{2}\delta-\delta]=[\chi_{1}\chi_{2}\delta-\chi_{2}\delta]+[\chi_{2}\delta-\delta]={\rm var}_{\chi_{1}}\mf d+{\rm var}_{\chi_{2}}\mf d$ car $[\chi_{2}\delta]=[\delta]=\mf d$ dans $H_{1}(\mc M_{\mc C},\partial \mc M_{\mc C};\mb Z)$.} ${\rm var}_{\chi_1\circ \chi_2} = {\rm var}_{\chi_1}+{\rm var}_{\chi_2}$, il suffit de d\'{e}terminer  $\Psi$ pour $L=L_k : [\delta]\mapsto \mf c_k$, $k=0,1$, o\`{u} $\mf c_0$ et $\mf c_1$ sont les m\'{e}ridiens associ\'{e}s \`{a} $D_0$ et \`{a} $D_1$. En effet ceux-ci forment une $\mb Z$-base de $H_1(\mc M_{\mc C},\mb Z)$, d'apr\`{e}s la proposition \ref{det}. Pour $k=0$ ou $1$, fixons  comme  (\ref{coordhomo})  des coordonn\'{e}es $(u, v)$
au point $s_0$ dans lesquelles l'application $f\circ E$ est monomiale et $v=0$ est une \'{e}quation de $D_k$.  L'hom\'{e}omorphisme (twist de Dehn) $\Psi : \mc T_{\mc C}\to \mc T_{\mc C}$ d\'{e}finit par
$$
u\circ \Psi=u\,,\quad v\circ\Psi = \left\{
  \begin{array}{ll}
    e^{2i\pi (3|u|-1)}\cdot v, & \hbox{\rm si } \frac{1}{3}\leq |u|\leq \frac{2}{3}\,, \\
    v, & \hbox{\rm si non}\,,
  \end{array}
\right.
$$
convient.\end{dem2}

\begin{dem2}{de la Proposition (\ref{eleinparticulier})}

La preuve consiste \`{a} appliquer convenablement le th\'{e}or\`{e}me de classification d'Eilenberg, cf. \cite[Theorem V.6.7]{Whitehead}, dont nous rappelons l'\'{e}nonc\'{e}  :
\begin{teo}
Soient $Y$ un espace topologique $(n-1)$-connexe avec $\pi=\pi_{n}(Y)$ ab\'{e}lien, $(X,A)$ un CW-complexe relatif et $f_{0}:X\to Y$ une application continue. Supposons que
\begin{enumerate}[(1)]
\item $Y$ est $q$-simple pour $n+1\le q\le\dim(X,A)$,
\item $H^{q}(X,A;\pi_{q}(Y))=0$ pour $n+1\le q\le\dim(X,A)$,
\item $H^{q+1}(X,A;\pi_{q}(Y))=0$ pour $n+1\le q\le\dim(X,A)-1$.
\end{enumerate}
Alors la correspondance $f\mapsto (f_{0},f)^*\imath^{n}(Y)$ induit une bijection entre l'ensemble de classes d'homotopie relatives \`{a} $A$ d'extensions de $f_{0|A}$ et le groupe de cohomologie $H^{n}(X,A;\pi)$.
\end{teo}
Dans cet \'{e}nonc\'{e}
$\imath^{n}(Y)\in H^{n}(Y;\pi)\cong\mr{Hom}(H_{n}(Y),\pi)$ s'identifie \`{a} l'inverse de l'isomorphisme de Hurewicz $\pi_{n}(Y)\stackrel{\sim}{\to}H_{n}(Y)$. Si $Y$ est un CW-complexe, alors
$\imath^{n}(Y)$	envoie chaque $n$-cellule de $Y$ sur l'unique \'{e}l\'{e}ment de $\pi=\pi_{n}(Y)$ obtenu en \'{e}crasant le $(n-1)$-squelette de $Y$ au point base. D'autre part
$(f_{0},f_{1})^*=(\mf i^*\times)^{-1}\circ\partial^*\circ F^*_{(f_{0},f_{1})}$, o\`{u} l'application
$F_{(f_{0},f_{1})}:X\times\partial\mb I\cup A\times\mb I\to Y$  est d\'{e}finie par $F_{(f_{0},f_{1})}(x,t)=f_{t}(x)$ si $x\in X$ et $t\in\partial\mb I=\{0,1\}$ et par  $F_{(f_{0},f_{1})}(a,t)=f_{0}(a)=f_{1}(a)$ si $a\in A$ et $t\in\mb I:=[0,1]$. Finalement $\partial^*:H^{n}(X\times\partial\mb I\cup A\times\mb I;\pi)\to H^{n+1}(X\times\mb I,X\times\partial\mb I\cup A\times\mb I;\pi)$ est le morphisme de connexion et
$\mf i^*\times:H^{n}(X,A;\pi)\to H^{n+1}(X\times\mb I,X\times\partial\mb I\cup A\times\mb I;\pi)$ est l'isomorphisme induit par le produit par le g\'{e}n\'{e}rateur $\mf i\in H^{1}(\mb I,\partial\mb I)$, en remarquant que $(X\times\mb I,X\times\partial\mb I\cup A\times\mb I)=(X,A)\times(\mb I,\partial\mb I)$.

\bigskip

Notons encore $\mb T=\mb S^{1}\times \mb S^{1}$ et $\mb I=[0,1]$. Si $\mc C$ est une cha\^{\i}ne, resp. une branche morte, nous appliquons le th\'{e}or\`{e}me avec $X=Y:= \mc M_{\mc C}$ qui est hom\'{e}omorphe \`{a} $\mb T\times\mb I$, resp. \`{a}  $X=Y\cong\mb D\times\mb S^{1}$, et est donc  est un espace de Eilenberg-MacLane $K(\pi,1)$, avec $\pi=\pi_{1}(\mb T\times\mb I)=H_{1}(\mb T\times\mb I)\cong\mb Z^{2}$ (resp. $\pi=\mb Z$). Les hypoth\`{e}ses du th\'{e}or\`{e}me pr\'{e}c\'{e}dent sont donc trivialement satisfaites. Nous posons aussi  $A:= \partial \mc M_{\mc C}\cong \mb T\times\partial\mb I$, resp. $A\cong\partial\mb D\times\mb S^{1}$ et $f_{0}=\mr{id}$.

Si $\mc C$ est une branche morte alors $$H^{1}(X,A;\pi)=H^{1}(\mb D\times\mb S^{1},\partial\mb D\times\mb S^{1},\mb Z)=H^{1}((\mb D,\partial\mb D)\times(\mb S^{1},\emptyset))
=0,$$ par la formule de K\"unneth relative
et par le fait que $H^{i}(\mb D,\partial\mb D)=0$ pour $i=0,1$. Dans ce cas on obtient donc que toutes les extension de l'identit\'{e} sur $A$ sont homotopes relativement \`{a} $A$.

Dans le cas o\`{u}
$\mc C$ est une cha\^{\i}ne (de $\mf C$ ou une paire de composantes associ\'{e}e \`{a} une transform\'{e}e stricte) nous obtenons que les classes d'homotopie relatives \`{a} $A$ d'extensions de l'identit\'{e} sont en correspondance bijective avec $H^{1}(\mb T\times\mb I,\mb T\times\partial\mb I;\mb Z^{2})\cong\mb Z^{2}$
Il suffit de montrer que  si $f:\mb T\times\mb I\to\mb T\times\mb I$ est une extension de l'identit\'{e} sur $\mb T\times\partial\mb I$ telle que
$\mr{var}_{f}=0$, alors $(\mr{id},f)^*\imath^{1}(\mb T\times\mb I)=(\mr{id},\mr{id})^*\imath^{1}(\mb T\times\mb I)$. En fait, pour ne pas avoir \`{a} travailler avec le morphisme de connexion, il suffit de voir que
$$F_{(\mr{id},f)}^*\imath^{1}(\mb T\times\mb I)=F_{(\mr{id},\mr{id})}^*\imath^{1}(\mb T\times\mb I)\in H^{1}(\mb T\times\mb I\times\partial\mb I\cup\mb T\times\partial\mb I\times\mb I;\mb Z^{2}).$$
Comme
$\mb T\times\mb I\times\partial\mb I\cup\mb T\times\partial\mb I\times\mb I=\mb T\times\partial(\mb I\times\mb I)$, on voit  que
$H_{1}(\mb T\times\mb I\times\partial\mb I\cup\mb T\times\partial\mb I\times\mb I)\cong H_{1}(\mb T)\oplus H_{1}(\partial(\mb I\times\mb I))\cong \mb Z^{3}$, d'o\`{u}
$H^{1}(\mb T\times\mb I\times\partial\mb I\cup\mb T\times\mb\partial I\times\mb I;\mb Z^{2})\cong\mr{Hom}(H_{1}(\mb T\times\mb I\times\partial\mb I\cup\mb T\times\partial\mb I\times\mb I),\mb Z^{2})\cong\mb Z^{3}\otimes\mb Z^{2}$.

Rappelons que $\mf c_{0},\mf c_{1}$ est une base de $H_{1}(\mb T\times \mb I)=H_{1}(\mb T)$ telle que $\mf c_{0}\subset\mb D^*\times\{e^{i\theta}\}$ et $\mf c_{1}\subset\{z\}\times\mb S^{1}$. Soit $\mf e$ un g\'{e}n\'{e}rateur de $H_{1}(\partial(\mb I\times\mb I))\cong\mb Z$.
Il est facile \`{a} voir que $\imath^{1}(\mb T\times\mb I)\in H^{1}(\mb T\times\mb I;\pi_{1}(\mb T\times\mb I))\cong \mr{Hom}(H_{1}(\mb T\times\mb I),H_{1}(\mb T\times\mb I))$ s'identifie a l'application identit\'{e} et alors $F_{(\mr{id},f)}^*\imath^{1}(\mb T\times\mb I)\cong F_{(\mr{id},f)*}$, o\`{u}
$$F_{(\mr{id},f)*}:H_{1}(\mb T\times\partial(\mb I\times\mb I))\cong\mb Z \mf c_{0}\oplus\mb Z \mf c_{1}\oplus\mb Z \mf e\to\mb Z \mf c_{0}\oplus\mb Z \mf c_{1}\cong H_{1}(\mb T\times\mb I)$$ s'identifie \`{a} une matrice de la forme
$$\left(\begin{array}{ccc}
1 & 0 & k\\
0 & 1 & m
\end{array}\right),$$ $(m,k)\in \mb Z^{2}$ v\'{e}rifient que
$\mr{var}_{f}(\mf d)=k \mf c_{0}+m \mf c_{1}$ o\`{u} $\mf d$ est le g\'{e}n\'{e}rateur de $H_{1}(\mb T\times\mb I,\mb T\times\mb\partial \mb I;\mb Z)\cong\mb Z$ qui joigne les deux composantes connexes de $\mb T\times\partial\mb I$. Ceci compl\`{e}te la preuve de la proposition.
\end{dem2}

\section{Groupe d'automorphismes d'un germe de courbe}\label{5}

\'{E}tant donn\'{e}e un germe de courbe plane $S$,  nous notons :
\begin{itemize}
\item $\mc G_{S}$ l'ensemble de marquages de $S$ par lui m\^{e}me, 
qui est un groupe pour la composition;
\item
 $\Gamma_{S}$ le groupe fondamental du tube de Milnor \'{e}point\'{e} $T_{\eta}\setminus S$;
\item
$\mr{Out}(\Gamma_{S}):=\mr{Aut}(\Gamma_{S})/\mr{Inn(\Gamma_{S})}$ le groupe d'automorphismes ext\'{e}rieurs de $\Gamma_{S}$;
\item $\mr{Out}_{g}(\Gamma_{S})$ le sous-groupe de $\mr{Out}(\Gamma_{S})$ form\'{e} des automorphismes ext\'{e}rieurs g\'{e}om\'{e}triques, 
cf. la d\'{e}finition \ref{isomgeom}.
\end{itemize}

\begin{teo}
L'application $*:\mc G_{S}\to \mr{Out}(\Gamma_{S})$, qui a chaque marquage $[h]$ associe son action $h_{*}$ sur le groupe fondamental $\Gamma_{S}$, est un isomorphisme sur $\mr{Out}_{g}(\Gamma_{S})$.
\end{teo}

\begin{proof}
L'application $*$ est bien d\'{e}finie pr\'{e}cis\'{e}ment parce qu'on consid\`{e}re des automorphismes ext\'{e}rieurs de $\Gamma_{S}$ qui \'{e}liminent l'ambigu\"{\i}t\'{e} du choix de $h$ dans la classe fondamental $[h]$. L'application $*$ est trivialement un morphisme de groupes injectif
gr\^{a}ce \`{a} la Proposition~\ref{fondequivhomot}
car $T_{\eta}\setminus S$ est un espace $K(\Gamma_{S},1)$. Finalement, la surjectivit\'{e} sur $\mr{Out}_{g}(\Gamma_{S})$ est cons\'{e}quence du corollaire \ref{invgeo}.
\end{proof}

\begin{cor}\label{ext}
Tout \'{e}l\'{e}ment de  $\mr{Out}_{g}(\Gamma_{S})$ est r\'{e}alisable par un hom\'{e}o\-mor\-phis\-me excellent de $(T_{\eta},S)$ sur lui m\^{e}me.
\end{cor}

Soit $\mb A_{S}$ l'arbre dual pond\'{e}r\'{e} de la r\'{e}solution minimale de $S$ et $\mf S_{S}$ le groupe de permutations des composantes irr\'{e}ductibles de $S$.
Il existe deux morphismes naturels bien d\'{e}finis $\sigma:\mc G_{S}\to\mf S_{S}$ et $\bar\sigma:\mr{Aut}(\mb A_{S})\to\mf S_{S}$.
L'existence d'un hom\'{e}omorphisme excellent dans chaque classe d'homotopie de $\mc G_{S}$ et l'injectivit\'{e} de $\bar\sigma$, prouv\'{e} dans le lemme suivant, permet de consid\'{e}rer un morphisme bien d\'{e}fini
$$\alpha:\mc G_{S}\to \mr{Aut}(\mb A_{S})$$
tel que $\sigma=\bar\sigma\circ\alpha$.

\begin{lema}\label{a} Avec les notations pr\'{e}c\'{e}dentes on a que :
\begin{enumerate}[(i)]
\item $\bar\sigma$ est injective, et en cons\'{e}quence $\ker\sigma=\ker\alpha$;
\item $\alpha$ est surjective et donc $\mr{Im}\,\sigma=\mr{Im}\,\bar\sigma$.
\end{enumerate}
\end{lema}

\begin{proof}
La premi\`{e}re assertion se d\'{e}montre facilement par r\'{e}currence sur le nombre $r$ de composantes irr\'{e}ductibles de $S$. Le cas $r=1$ se d\'{e}montre par r\'{e}currence sur le nombre $g$ de paires de Puisseux de $S$. Quand $g=1$ une description tr\`{e}s explicite de la situation permet de montrer que  $\bar\sigma$ est injective dans ce cas.
%
La deuxi\`{e}me assertion se d\'{e}montre aussi par r\'{e}currence sur le nombre de composantes irr\'{e}ductibles de $S$. Quand $S$ est irr\'{e}ductible $\mr{Aut}(\mb A_{S})=\{\mr{id}\}$ d'apr\`{e}s (i).
Si $S_{i}$ et $S_{j}$
sont deux composantes irr\'{e}ductibles de $S$ exchang\'{e}es par $g\in\mr{Aut}(\mb A_{S})$ alors les sous-arbres pond\'{e}r\'{e}s correspondants aux r\'{e}ductions de $S_{i}$ et $S_{j}$ sont isomorphes. Dans ce cas il est facile \`{a} voir qu'il existe un hom\'{e}omorphisme de $(\mc T_{\eta},\mc D)$ qui induit $g$ et qui est l'identit\'{e} en dehors d'un voisinage de de la partie du diviseur $\mc D$ qui n'intersecte pas les sous-arbres correspondants \`{a} $S_{i}$ et $S_{j}$.
\end{proof}

Toujours avec les notations (\ref{KD}), (\ref{XK}), (\ref{4blocselem}), (\ref{4voischaine}), pour une cha\^{\i}ne $\mc C\in \mf C$ nous posons :  $K_{\mc C}:=\mc T_{\mc C}\cap \mc D$,
$\mc T_{\eta}(\partial K_{\mc C})=\mc T_{\eta}(\partial (K_{\mc C}\cap D_{0}))\cup\mc T_{\eta}(\partial (K_{\mc C}\cap D_{l_{\mc C}+1}))$.
\begin{defin}
%
Pour chaque \'{e}l\'{e}ment $B\in\mf B:=\mf R\cup\mf C$
consid\'{e}rons
le groupe $\mc G_{B}$ des classes d'homotopie relatives \`{a} $K_{B}\cup\mc T_{\eta}(\partial K_{B})$ d'hom\'{e}o\-mor\-phis\-mes de 
$\mc T_{B}$ qui laissent $K_{B}$ invariant et sont l'identit\'{e}  sur $\mc T_{\eta}(\partial K_{B})$.
\end{defin}

Tout \'{e}l\'{e}ment de $\mc G_{B}$ induit un marquage excellent \`{a} support contenu dans $\mc T_{\eta}(K_{B})$. Nous avons donc un morphisme bien d\'{e}fini
$$\beta:\bigoplus\limits_{B\in\mf B}\mc G_{B}\to\mc G_{S}.$$

\begin{prop}\label{GB}
Consid\'{e}rons $D\in\mf R$ et $\mc C\in\mf C$.
\begin{enumerate}[(i)]
\item Le groupe $\mc G_{D}$ est isomorphe au groupe $A(D^{\bullet})$ des classes d'homotopie relative \`{a} $D\cap\mr{Sing}(\mc D)$ d'hom\'{e}omorphismes de $D$ fixant chaque point de $S(D)$.
\item Tout \'{e}l\'{e}ment de $\mc G_{\mc C}$ est un twist de Dehn le long de $\mc C$, cf.~section~(\ref{dehn}). En particulier, $\mc G_{\mc C}\cong\mb Z^{2}$.
\end{enumerate}
\end{prop}

\begin{proof} Pour prouver (i) on trivialise $\mc T_{\eta}(K_{D})\cong K_{D}\times\mb D$
et on \'{e}crit un repr\'{e}sentant excellent d'un \'{e}l\'{e}ment quelconque $\mf f$ de $\mc G_{D}$  sous la forme $(f,g)$ o\`{u} $f:K_{D}\to K_{D}$  est un hom\'{e}omorphisme valant l'identit\'{e} sur $\partial K_{D}$ et $g:K_{D}\to \mr{Homeo}(\mb D,0)\simeq \mb S^{1}$. Comme $g_{|\partial K_{D}}$ est constante \'{e}gal \`{a} $\mr{id}_{\mb D}$ il en r\'{e}sulte que $(f,g)$ est isotope \`{a} $(f,\mr{id}_{\mb D})$. Ainsi, $\mf f=[(f,g)]$ est compl\`{e}tement d\'{e}termin\'{e} par $[f]\in A(D^{\bullet})$. R\'{e}ciproquement, tout \'{e}l\'{e}ment $[f]\in A(D^{\bullet})$ d\'{e}termine de fa\c{c}on univoque l'\'{e}l\'{e}ment $[(f,\mr{id}_{\mb D})]\in\mc G_{D}$.
D'autre part, l'assertion (ii) est une cons\'{e}quence imm\'{e}diate de la Proposition~\ref{eleinparticulier}.
\end{proof}

Le groupe modulaire pur $A(D^{\bullet})$ s'identifie au quotient du groupe d'Artin de tresses pures du plan \`{a} $v(D)-1$ brins, par son centre qui est isomorphe a $\mb Z$. Il s'identifie aussi au quotient du groupe de tresses pures de la sph\`{e}re a $v(D)$ brins, par son centre qui est isomorphe \`{a} $\mb Z/2\mb Z$, voir par exemple \cite{Birman}. Nous appellerons les \'{e}l\'{e}ments de $\mc G_{D}$ \emph{twists d'Artin au dessus de $D$}.

\medskip

D'apr\`{e}s la Proposition~\ref{GB}, le th\'{e}or\`{e}me~\ref{thmB} de l'introduction affirme que l'image de $\beta$ est le noyau $\mc G_{S}^{0}$ de $\sigma$, c'est  \`{a} dire, les twists d'Artin et les twists de Dehn engendrent le sous-groupe d'indice fini $\mc G_{S}^{0}$ de $\mc G_{S}$.






\begin{proof}[Preuve du th\'{e}or\`{e}me \ref{thmB}]
D'apr\`{e}s le th\'{e}or\`{e}me principal et le lemme \ref{a}, tout \'{e}l\'{e}ment de $\mc G_{S}^{0}$ peut \^{e}tre repr\'{e}sent\'{e} par un hom\'{e}omorphisme excellent $f:\mc T_{\eta}\to\mc T_{\eta}$ qui fixe chaque composante irr\'{e}ductible de $\mc D$ et qui est l'identit\'{e}\footnote{Ceci est possible gr\^{a}ce \`{a} l'holomorphie de $f$ au voisinage de chaque singularit\'{e} de $\mc D$.} sur le bord de chaque bloc $\mc T_{\eta}(K_{D})$ et $\mc T_{\eta}(\mc C)$.
D'apr\`{e}s le th\'{e}or\`{e}me de Seifert-Van Kampen, le groupe fondamental
$\Gamma_{S}$ est le produit amalgam\'{e} des groupes fondamentaux $\Gamma_{S}(D)=\pi_{1}(B_{D})$, $D\in\mf R$, des blocs Seifert de la d\'{e}composition JSJ de $M_{\eta}$ au dessus des groupes fondamentaux de ses tores essentiels $\Gamma_{S}(\mc C)=\pi_{1}(\mb T_{\mc C})$, $\mc C\in\mf C$.
Soit $D\in\mf R$ un sommet terminal de l'arbre  de JSJ
de $M_{\eta}$, cf. (\ref{grapheJSJ}), et $\mc C\in\mf C$ sa cha\^{\i}ne adjacente.
En composant $f$ par deux \'{e}l\'{e}ments convenables de $\mc G_{D}$ et $\mc G_{\mc C}$ on peut supposer que $f_{*}:\Gamma_{S}\to\Gamma_{S}$
est l'identit\'{e} sur $\Gamma_{S}(D)\supset\Gamma_{S}(\mc C)$. On conclut en raisonnant par r\'{e}currence sur le nombre de blocs de Seifert sur lesquels $f_{*}$ n'est pas l'identit\'{e}.
\end{proof}

L'exemple suivant montre que l'\'{e}pimorphisme
du th\'{e}or\`{e}me~\ref{thmB} 
n'est pas en g\'{e}n\'{e}ral injectif. Ainsi il peut exister d'autres relations  entre les g\'{e}n\'{e}rateurs de $\mc G_{D}$ et $\mc G_{\mc C}$ \`{a} part de celles qu'on vient d'expliciter.

\begin{ex}
La courbe $S=f^{-1}(0)$ avec $f(x,y)=y(y^{2}-x^{3})^{2}-x^{8}$ a deux paires de Puiseux, le diviseur exceptionnel de sa d\'{e}singularisation minimale consiste en cinq droites $E_{i}$, $i=1,\ldots,5$, num\'{e}rot\'{e}s par ordre d'apparition et ayant pour matrice d'intersection $$\left(\begin{array}{ccccc}
-3 & 0 & 1 & 0 & 0\\
0 & -2 & 1 & 0 & 0\\
1 & 1 & -3 & 0 & 1\\
0 & 0 & 0 & -2 & 1\\
0 & 0 & 1 & 1 & -1
\end{array}\right).$$
Dans ce cas, il y a deux diviseurs de valence trois $E_{3}$ et $E_{5}$ avec deux (resp. une) branches mortes adjacentes $E_{1}$, $E_{2}$ (resp. $E_{4}$). Il n'y a qu'une cha\^{\i}ne $\mc C$ de longueur $0$ correspondant au point $E_{3}\cap E_{5}$.
Le groupe fondamental $\Gamma_{S}$ d'un tube de Milnor de $f$ moins $S$ admet comme syst\`{e}me de g\'{e}n\'{e}rateurs
les classes d'homotopie $a_{1},b_{1},c_{1},b_{2},c_{2},d$ de lacets contenus dans des fibres de Hopf des diviseurs $E_{1},E_{2},E_{3},E_{4},E_{5}$ et $S$ respectivement.
Les relations de ces g\'{e}n\'{e}rateurs sont engendr\'{e}s par
$$a_{1}^{3}=c_{1}=b_{1}^{2},\quad a_{1}b_{1}c_{2}=c_{1}^{3},\quad  c_{2}=b_{2}^{2},\quad  c_{1}b_{2}d=c_{2}$$
et
\begin{equation}\label{comm}[c_{1},a_{1}]=[c_{1},b_{1}]=[c_{1},c_{2}]=[c_{2},b_{2}]=[c_{2},d]=1.
\end{equation}
En prenant le point de base convenablement,
l'action sur $\Gamma_{S}$  d'un twist de Dehn autour de $\mc C$ de type $(p,q)$ est de la forme suivante :
$$a_{1}\mapsto a_{1},\quad b_{1}=b_{1},\quad c_{1}\mapsto c_{1},\quad  b_{2}\mapsto c_{1}^{p}b_{2}c_{1}^{-p},\quad c_{2}\mapsto c_{2},\quad d\mapsto c_{1}^{p}dc_{1}^{-p},$$
qui d'apr\`{e}s les relations (\ref{comm})
co\"{\i}ncide avec l'automorphisme int\'{e}rieur associ\'{e} \`{a}  l'\'{e}l\'{e}ment $c_{1}^{p}c_{2}^{q}\in\Gamma_{S}$. Ainsi, dans ce cas,  $\beta(\mc G_{\mc C})\subset\ker(*)$ qui est trivial.
\end{ex}

\end{document}